\documentclass[10pt]{amsart}
\usepackage{hyperref}
\usepackage{amsmath}
\usepackage{amssymb}
\usepackage{xypic}
\usepackage{yhmath}
\usepackage{setspace}
\def\beqnn{\begin{eqnarray*}}\def\eeqnn{\end{eqnarray*}}

\newtheorem{theorem}{Theorem}[section]
\newtheorem{theorem*}{Theorem}
\newtheorem{lemma}[theorem]{Lemma}

\theoremstyle{definition}    % 定义样式：粗体标题，正体正文

\newtheorem{remark}[theorem]{Remark}

\numberwithin{equation}{section}

\begin{document}

\begin{center}
\title[Hardy-Littlewood-P\'{o}lya-type operators]{$l^{p}-L^{q}$ boundedness of sequence-to-function Hardy-Littlewood-P\'{o}lya-type operators}
\end{center}

\author{Jianjun Jin }
\address{School of Mathematics Sciences, Hefei University of Technology, Xuancheng Campus, Xuancheng 242000, P.R.China}
\email{jin@hfut.edu.cn, jinjjhb@163.com}
%\thanks{The author was supported by National Natural Science Foundation of China (Grant Nos. 11501157).}
%\author{Huabing Li}
%\address{School of Mathematics Sciences, Hefei University of Technology, Xuancheng Campus, Xuancheng 242000, P.R.China}
%\email{musicli121@163.com}
%\thanks{The authors were supported by National Natural Science Foundation of China (Grant Nos. 11501157).}
%\author{Author Two}
%\address{Mathematical Research Section, School of Mathematical Sciences,
%Australian National University, Canberra ACT 2601, Australia}

%    General info

%\date{February 9, 2013.}


\begin{abstract} 
In this paper we deal with the boundedness of certain generalized sequence-to-function Hardy-Littlewood-P\'{o}lya-type operators. We completely characterize the $l^{p}-L^{q}$ boundedness of these operators for all $(p, q)\in [1, \infty]\times[1, \infty]$. These results complement some previous results in the literature. Our method relies on the ideas of generalized Schur’s tests developed by Okikiolu, Sinnamon, Tao and Zhao.  
\end{abstract} 
\maketitle

\section{\bf {Introduction and main results} }
Let $p\geq 1$, we use $p'$ to denote the conjugate of $p$,  i.e., $\frac{1}{p}+\frac{1}{p'}=1$. In particular, $p'=\infty$ when $p=1$. Let $\mathbb{R}_{+}=(0,+\infty)$ and let $\mathcal{M}(\mathbb{R}_{+})$ be the class of all real-valued measurable functions on $\mathbb{R}_{+}.$ Let $L^p:=L^p(\mathbb{R}_{+})$ be the usual Lebesgue space on $\mathbb{R}_{+}$, i.e.,
\begin{equation*}
L^p(\mathbb{R}_{+})=\{f\in \mathcal{M}(\mathbb{R}_{+}): \|f\|_p=(\int_{\mathbb{R}_{+}}
|f(x)|^pdx)^{\frac{1}{p}}<\infty\}.
\end{equation*}
When $p>1$. If $f \in L^p, g \in L^{p'}$,  then we have the famous Hardy-Littlewood-P\'{o}lya (${\mathcal{HLP}}$) inequality as
\begin{equation}\label{HLP}
\Big|\int_{\mathbb{R}_{+}}
\int_{\mathbb{R}_{+}}\frac{f(x)g(y)}{\max\{x, y\}}dxdy\Big|\leq (p+p')\|f\|_{p}\|g\|_{p'},
\end{equation}
where the constant $p+p'$ in (\ref{HLP}) is the best possible, see \cite{HLP}. (\ref{HLP}) has the following equivalent form
\begin{equation}\label{HLP-1}
\left[\int_{\mathbb{R}_{+}}
\left|\int_{\mathbb{R}_{+}}\frac{f(y)}{\max\{x,y\}}dy\right|^p
dx\right]^{\frac{1}{p}}\leq (p+p')\|f\|_{p}.
\end{equation}
(\ref{HLP-1}) is also called ${\mathcal{HLP}}$ inequality and the constant $p+p'$ in (\ref{HLP-1}) is still the best possible.  

We can restate the ${\mathcal{HLP}}$ inequality (\ref{HLP-1}) in the language of operator theory. We define the ${\mathcal{HLP}}$ integral operator $H$ induced by ${\mathcal{HLP}}$ kernel $\frac{1}{\max\{x,y\}}$ as 
\begin{equation*}Hf(y):=\int_{\mathbb{R}_{+}} \frac{f(x)}{\max\{x,y\}}dx,\: f \in \mathcal{M}(\mathbb{R}_{+}),\: y\in \mathbb{R}_{+}.
\end{equation*}
Then we have
\begin{theorem}\label{fir}Let $p>1$. Let $H$ be as above. Then $H$ is bounded on $L^p(\mathbb{R}_+)$ and the norm $\|H\|_{L^p\rightarrow L^p}$ of $H$ is $pp'$. Here
\begin{equation*}
\|H\|_{L^p\rightarrow L^p}=\sup_{f\in L^p(\mathbb{R}_{+})}
\frac{\|Hf\|_{p}}{\|f\|_{p}}.
\end{equation*}
\end{theorem}
There is also a discrete version of Theorem \ref{fir}.  For a sequence $a=\{a_m\}_{m=1}^{\infty}$, we define ${\mathcal{HLP}}$ operator $\mathcal{H}$ as
$${\mathcal{H}}a(n):=\sum_{m=1}^{\infty}\frac{a_m}{\max\{m,n\}}, \,n\in \mathbb{N}.$$
We use $l^p$ to denote the space of sequences of real numbers, i.e.,
\begin{equation*}l^{p}:=\{a=\{a_n\}_{n=1}^{\infty}: \|a\|_{p}= (\sum_{n=1}^{\infty} |a_{n}|^p )^{\frac{1}{p}}<\infty \}.\end{equation*}
Then we have
\begin{theorem}Let $p>1$. Let $\mathcal{H}$ be as above. Then $\mathcal{H}$ is bounded on $l^p$ and the norm $\|\mathcal{H}\|_{l^p \rightarrow l^p}$ of $\mathcal{H}$ is $pp'$. Here
\begin{equation*}
\|\mathcal{H}\|_{l^p \rightarrow l^p}=\sup_{a\in l^p}
\frac{\|\mathcal{H}a\|_{p}}{\|a\|_{p}}.
\end{equation*}
\end{theorem}

The Hardy-Littlewood-P\'olya operator is related to some important topics in analysis and there have been many results about this operator and its analogues and generalizations. The classical results of this operator can be found in the well-known monograph \cite{HLP}.  In the past three decades,  the so-called Hilbert-type operators, including ${\mathcal{HLP}}$-type operators, have been extensively studied by Yang and his coauthors,  see the survey \cite{YR} and Yang's book \cite{Y3}. For more recent results see for example \cite{Jin}, \cite{WHY}, \cite{YZ}. In the work \cite{B}, Brevig established some norm estimates for certain ${\mathcal{HLP}}$-type operators in terms of Riemann zeta function.  Some further results can be found in \cite{B-1}. 

Recently, much attention has been paid to the study of the boundedness and estimate for the norm of the following operator $H_K$, which is defined as
\begin{equation*} H_{K}a(x):=\sum_{m=1}^{\infty}K(m, x)a_m ,\, \: a=\{a_m\}_{m=1}^{\infty}, x\in \mathbb{R}_{+}.\end{equation*}
where $K(y,x)$ is a non-negative measurable function on $\mathbb{R}_{+}\times \mathbb{R}_{+}$. For more detailed  introductions to  this topic, see the book \cite{YL} of Yang and Lokenath, and the references cited therein. In particular, take $K(y, x)=[\max\{y,x\}]^{-1}$, we get the following operator
\begin{equation*}{\bf H}a(x):=\sum_{m=1}^{\infty}\frac{a_m}{\max\{m,x\}},\, \: a=\{a_m\}_{m=1}^{\infty}, x\in \mathbb{R}_{+}.\end{equation*}

In this paper, we continue the study of this topic. Let $\lambda, \alpha, \beta$ be real numbers. For $a=\{a_m\}_{m=1}^{\infty}$, we define the operator ${\bf H}_{\lambda, \alpha, \beta}$ as 
\begin{equation}\label{ope}{\bf H}_{\lambda, \alpha, \beta}a(x):=\sum_{m=1}^{\infty}\frac{m^{\alpha}x^{\beta}}{[\max\{m, x\}]^{\lambda}}a_m,\,\,\,  x\in \mathbb{R}_{+}.\end{equation}
We call ${\bf H}_{\lambda, \alpha, \beta}$ {\em sequence-to-function ${\mathcal{HLP}}$-type operator}. We will investigate the conditions of the boundedness of $H_{\lambda, \alpha, \beta}$. Note that previous studies have primarily focused on the $l^{p}-L^{q}$ boundedness of sequence-to-function ${\mathcal{HLP}}$-type operators for $p=q>1$. In the present paper, we shall completely characterize the $l^{p}-L^{q}$ boundedness of these operators for all $(p, q)\in [1, \infty]\times[1, \infty]$. Our method relies on the ideas of generalized Schur’s tests, see \cite{O}, \cite{S}, \cite{Tao} and \cite{Zh}. The method used in this paper can be applied to other non-negative measurable kernel $K(y,x)$ defined on $\mathbb{R}_{+}\times \mathbb{R}_{+}$.  

To state our results, we introduce some notation. For $1\leq p<\infty, \theta\in \mathbb{R}$, we define the weighted Lebesgue space 
$L_{\theta}^p:= L_{\theta}^p(\mathbb{R}_{+})$ on $\mathbb{R}_{+}$ as 
\begin{displaymath} 
L_{\theta}^p(\mathbb{R}_{+})=\{f\in \mathcal{M}(\mathbb{R}_{+}): 
||f||_{p,\theta}=(\int_{\mathbb{R}_{+}} 
|f(x)|^p x^{\theta}dx)^{\frac{1}{p}}<\infty \}. 
\end{displaymath} 
We write $L^p$ and $\|f\|_{p}$ instead of $L_{\theta}^p$ and $\|f\|_{p,\theta}$, respectively, if $\theta=0$. 
We define the class $L^{\infty}:=L^{\infty}(\mathbb{R}_{+})$ as 
\begin{displaymath} 
L^{\infty}(\mathbb{R}_{+})=\{f\in \mathcal{M}(\mathbb{R}_{+}): \|f\|_{\infty}:=\mathop{\text {ess sup}}\limits_{x \in \mathbb{R}_{+}}|f(x)|<\infty\}.
\end{displaymath}
Meanwhile, we define the weighted space $l_{\theta}^p$ of sequences of real numbers as 
\begin{displaymath} 
l_{\theta}^p=\{a=\{a_n\}_{n=1}^{\infty}: ||a||_{p,\theta}=(\sum_{n=1}^{\infty}n^{\theta}|a_n|^p)^{\frac{1}{p}}<\infty \}. 
\end{displaymath} 
We use $l^p$ and $\|a\|_{p}$ to denote $l_{\theta}^p$ and $\|a\|_{p,\theta}$, respectively, if $\theta=0$. The class $l^{\infty}$ is defined as 
\begin{displaymath} 
l^{\infty}=\{a=\{a_n\}_{n=1}^{\infty}: \|a\|_{\infty}=\sup\limits_{n \in \mathbb{N}}|a_n|<\infty\}.
\end{displaymath} 

We will prove the following theorems for the boundedness of $\mathbf{H}_{\lambda, \alpha, \beta}$ for the cases $1\leq p\leq q \leq \infty$. 

\begin{theorem}\label{m-th-1}
Let $1\leq p\leq q<\infty$. Let $\lambda, \alpha, \beta, \varphi, \psi$ be real numbers and $\mathbf{H}_{\lambda, \alpha, \beta}$ be as in (\ref{ope}). Then $\mathbf{H}_{\lambda, \alpha, \beta}$ is bounded from
$l_{\varphi}^p$ to $L_{\psi}^q$ if and only if 
\begin{equation} 
\begin{cases}
\lambda\geq \alpha+\beta+1+\frac{\psi+1}{q}-\frac{\varphi+1}{p}, \\
-q\beta<\psi+1<q(\lambda-\beta). \nonumber 
\end{cases}
\end{equation}
\end{theorem}

\begin{theorem}\label{m-th-2}
Let $\lambda, \alpha, \beta$ be real numbers and ${\bf H}_{\lambda, \alpha, \beta}$ be as in (\ref{ope}). Then ${\bf H}_{\lambda, \alpha, \beta}$ is bounded from
$l_{\varphi}^1$ to $L^{\infty}$ if and only if 
\begin{equation}\label{th-2-1}
\begin{cases}
\lambda\geq \beta\geq 0, \\
\lambda\geq \alpha+\beta-\varphi.\nonumber
\end{cases}
\end{equation}
\end{theorem}

\begin{theorem}\label{m-th-3}
Let $1<p<\infty$. Let $\lambda, \alpha, \beta, \varphi$ be real numbers and ${\bf H}_{\lambda, \alpha, \beta}$ be as in (\ref{ope}). Then ${\bf H}_{\lambda, \alpha, \beta}$ is bounded from
$l_{\varphi}^p$ to $L^{\infty}$ if and only if 
%$$ \Big\{\begin{array}{cc} 
%\lefteqn{\lambda\geq \beta>0,} \\
%\lambda\geq \alpha+\beta+1-\frac{\varphi+1}{p}, 
%\end{array}\,\, or\,\, \Big\{ \begin{array}{cc} 
%\lambda\geq \beta=0, \\
%\lambda>\alpha+1-\frac{\varphi+1}{p}. \nonumber
%\end{array}$$
\begin{equation} 
\begin{cases}
\lambda>\beta\geq 0, \\
\lambda\geq \alpha+\beta+1-\frac{\varphi+1}{p}, \nonumber
\end{cases}
\end{equation}
or
\begin{equation} 
\begin{cases}
\lambda\geq \beta\geq 0, \\
\lambda>\alpha+\beta+1-\frac{\varphi+1}{p}. \nonumber
\end{cases}
\end{equation}
\end{theorem}

\begin{theorem}\label{m-th-3-1}
Let $\lambda, \alpha, \beta$ be real numbers and ${\bf H}_{\lambda, \alpha, \beta}$ be as in (\ref{ope}). Then ${\bf H}_{\lambda, \alpha, \beta}$ is bounded from
$l^\infty$ to $L^{\infty}$ if and only if 
\begin{equation}
\begin{cases}
\lambda\geq \beta\geq 0, \\
\lambda>\alpha+\beta+1. \nonumber
\end{cases}
\end{equation}
\end{theorem}

For the boundedness of $\mathbf{H}_{\lambda, \alpha, \beta}$ for the cases $1\leq q<p\leq \infty$, we shall prove that  

\begin{theorem}\label{th-r-10}
Let $1<p<\infty$. Let $\lambda, \alpha, \beta, \varphi, \psi$ be real numbers and $\mathbf{H}_{\lambda, \alpha, \beta}$ be as in (\ref{ope}).  Then  $\mathbf{H}_{\lambda, \alpha, \beta}$ is bounded from
$l_{\varphi}^p$ to $L_{\psi}^1$ if and only if 
\begin{equation} 
\begin{cases}
\lambda>\alpha+\beta+2+\psi-\frac{\varphi+1}{p}, \\
-\beta<\psi+1<\lambda-\beta. \nonumber 
\end{cases}
\end{equation}
\end{theorem}

\begin{theorem}\label{th-r-1}
Let $1<q<p<\infty$. Let $\lambda, \alpha, \beta, \varphi, \psi$ be real numbers and $\mathbf{H}_{\lambda, \alpha, \beta}$ be as in (\ref{ope}). If $\varphi+1<p(\alpha+1)$, then $\mathbf{H}_{\lambda, \alpha, \beta}$ is bounded from
$l_{\varphi}^p$ to $L_{\psi}^q$ if and only if 
\begin{equation} 
\begin{cases}
\lambda>\alpha+\beta+1+\frac{\psi+1}{q}-\frac{\varphi+1}{p}, \\
-q\beta<\psi+1. \nonumber 
\end{cases}
\end{equation}
\end{theorem}

\begin{theorem}\label{th-r-2}
Let $1\leq q<\infty$. Let $\lambda, \alpha, \beta, \psi$ be real numbers with $\alpha+1\geq0$, and ${\bf H}_{\lambda, \alpha, \beta}$ be as in (\ref{ope}).
Then ${\bf H}_{\lambda, \alpha, \beta}$ is bounded from $l^{\infty}$ to $L_{\psi}^{q}$ if and only if
\begin{equation} 
\begin{cases}
\lambda>\alpha+\beta+1+\frac{\psi+1}{q}, \\
-q\beta<\psi+1. \nonumber 
\end{cases}
\end{equation}
\end{theorem}

The rest of this paper is organized as follows.  We will establish some lemmas in the next section. We shall give the proof of Theorem \ref{m-th-1} in Section 3. We prove Theorem \ref{m-th-2}, \ref{m-th-3} in Section 4.  In Section 5, we prove Theorem \ref{m-th-3-1}.  We prove Theorem \ref{th-r-10}, \ref{th-r-1} and \ref{th-r-2} in Section 6. In Section 7, we shall give sharp norm estimates for $\mathbf{H}_{\lambda, \alpha, \beta}$ for some special cases.
We present two final remarks in Section 8.

\section{{\bf Lemmas}}
In this section, we list some known lemmas and establish some new ones.  We first recall the following two lemmas, see \cite[Problem 5.5]{Tao} for more general versions. 
\begin{lemma}\label{ll-1}
Let $1\leq p <\infty$. Let $K(y,x)$ be a measurable function on $\mathbb{R}_{+}\times \mathbb{R}_{+}$ with $0 \leq K(y,x)<\infty$ for all $(y,x)\in \mathbb{R}_{+}\times \mathbb{R}_{+}$. For $a=\{a_m\}_{m=1}^{\infty}$, let $T$ be the operator with kernel $K$ defined as
$$Ta(x)=\sum_{m=1}^{\infty}K(m,x)a_m,\,\,\, x\in \mathbb{R}_{+}.$$
Then $T$ is bounded from $l^{p}$ to $L^{1}$ if and only if 
$$\int_{\mathbb{R}_{+}}K(m,x)dx\in l^{p'}.$$
Moreover, when $T$ is bounded from $l^{p}$ to $L^{1}$, the norm of $T$ is given by
$$\|T\|_{l^{p}\rightarrow L^1}=\|\int_{\mathbb{R}_{+}}K(m,x)dx\|_{p'}.$$
\end{lemma}
%\begin{remark}In particular, when $p=1$ so that $p'=\infty$, $T$ is bounded from $l^{1}$ to $L^{1}$ if and only if 
%$$\sup_{m\in \mathbb{N}}\int_{\mathbb{R}_{+}}K(m,x)dx<\infty.$$ 
%\end{remark}
The following is a dual version of Lemma \ref{ll-1}. 

\begin{lemma}\label{ll-2}
Let $1\leq p\leq\infty$. Let $K(y,x)$ be a measurable function on $\mathbb{R}_{+}\times \mathbb{R}_{+}$ with $0 \leq K(y,x)<\infty$ for all $(y,x)\in \mathbb{R}_{+}\times \mathbb{R}_{+}$. For $a=\{a_m\}_{m=1}^{\infty}$, let $T$ be the operator with kernel $K$ defined as
$$Ta(x)=\sum_{m=1}^{\infty}K(m,x)a_m,\,\,\, x\in \mathbb{R}_{+}.$$
Then $T$ is bounded from $l^{\infty}$ to $L^{p}$ if and only if 
$$\sum_{m=1}^{\infty}K(m,x)\in L^{p}.$$
Moreover, when  $T$ is bounded from $l^{\infty}$ to $L^{p}$, the norm of $T$ is given by
$$\|T\|_{l^{\infty}\rightarrow L^{p}}=\|\sum_{m=1}^{\infty}K(m,x)\|_{p}.$$
\end{lemma}
\begin{remark}In particular, when $p=\infty$, $T$ is bounded from $l^{\infty}$ to $L^{\infty}$ if and only if 
$$\sup_{x\in \mathbb{R}_{+}}\sum_{m=1}^{\infty}K(m,x)<\infty.$$
\end{remark}

We also need the following generalized Schur’s test, see \cite{S} for its general form. 
\begin{lemma}\label{ll-3}
Let $1<q\leq p <\infty$. Let $K(y,x)$ be a measurable function on $\mathbb{R}_{+}\times \mathbb{R}_{+}$ with $0 \leq K(y,x)<\infty$ for all $(y,x)\in \mathbb{R}_{+}\times \mathbb{R}_{+}$. For $a=\{a_m\}_{m=1}^{\infty}$, let $T$ be the operator with kernel $K$ defined as
$$Ta(x)=\sum_{m=1}^{\infty}K(m,x)a_m,\,\,\, x\in \mathbb{R}_{+}.$$
If there exist two constants $C_1>0, C_2>0$ and a sequence $u=\{u_n\}_{n=1}^{\infty}$ with $|u_n| \leq C_1$ for all $n\in \mathbb{N}$ and satisfying $I_{K}u(n) \leq C_2 u(n)$ for all $n\in \mathbb{N}$, then $T$ is bounded from $l^{p}$ to $L^{q}.$ Here
$$I_{K}u(n)=\Big(\int_{\mathbb{R}_{+}}K(n,x)\Big(\sum_{m=1}^{\infty}K(m,x)u_m\Big)^{q-1}dx\Big)^{p'-1},\,\, n\in \mathbb{N}.$$
\end{lemma}

We will use the following lemmas in our later arguments.

\begin{lemma}\label{ll-4}Let $\tau, \lambda$ be two real numbers and let $t>0$ be fixed. Then the integral 
$$I(t)=\int_{\mathbb{R}_{+}}\frac{x^{\tau}}{[\max\{t, x\}]^{\lambda}}dx,$$
converges if and only if $\tau>-1$ and $\lambda-\tau-1>0$. When the integral converges, we have
$$I(t)=t^{\tau+1-\lambda}\Big(\frac{1}{\tau+1}+\frac{1}{\lambda-\tau-1}\Big).$$ 
\end{lemma}

\begin{proof}
For fixed $t>0$, note that 
\begin{equation}\label{m-e}
I(t)=t^{-\lambda}\int_{0}^{t}x^{\tau}dx+\int_{t}^{\infty}x^{\tau-\lambda}dx.
\end{equation}
Then we conclude that $I(t)<\infty$ if and only if $\tau+1>0$ and $\lambda-\tau-1>0$. On the other hand, when $\tau+1>0$ and $\lambda-\tau-1>0$, it follows from (\ref{m-e}) that 
\begin{eqnarray} 
I(t)=t^{\tau+1-\lambda}\Big(\frac{1}{\tau+1}+\frac{1}{\lambda-\tau-1}\Big).\nonumber 
\end{eqnarray} 
The lemma is proved. 
\end{proof}

\begin{lemma}\label{ll-5}Let $\tau, \lambda$ be real numbers and $x>0$ be fixed. Then 
 $$S(x)=\sum_{m=1}^{\infty}\frac{m^{\tau}}{[\max\{m,x\}]^{\lambda}}<\infty$$
 if and only if $\lambda-\tau-1>0$. Moreover, {\bf (1)} when $\lambda-\tau-1>0$ and $-1<\tau\leq 0$, we have
$$S(x)\leq \Big(\frac{1}{\tau+1}+\frac{1}{\lambda-\tau-1}\Big)x^{\tau+1-\lambda},\,\,\,\, x>0,$$

{\bf (2)}
when $\lambda-\tau-1>0$ and $\tau>0$, we have
$$S(x)\leq C x^{\tau+1-\lambda},\,\,\, x>0,$$

%{\bf (3)} when $\lambda-\tau-1>0$ and $\tau=-1$, we have $$S(x)\leq C_2 x^{-\lambda-\epsilon},\,\, {\text for}\,\, x>1,$$ for any $\epsilon>0$,

%{\bf (4)} when $\lambda-\tau-1>0$ and $\tau<-1$, we have $$S(x)\leq C_3 x^{-\lambda},\,\, {\text for}\,\, x>1.$$ 
Here, $C$ is a positive constant which is independent of $x$. 
\end{lemma}

\begin{proof}
For fixed $x>0$, we note that 
$$\frac{1}{m^{\tau-\lambda}}\cdot\frac{m^{\tau}}{[\max\{m,x\}]^{\lambda}}\rightarrow 1,\,\, {\text{as}}\,\, m\rightarrow \infty.$$ 
Then we see that $S(x)<\infty$ if and only if $\lambda-\tau>1$. Moreover, 

Case 1. When $\lambda-\tau-1>0$ and $-1<\tau\leq 0$, we obtain from Lemma \ref{ll-4} that 
\begin{eqnarray}
S(x)\leq \int_{\mathbb{R}_{+}}\frac{t^{\tau}}{[\max\{{t,x}\}]^{\lambda}}dt=t^{\tau+1-\lambda}\Big(\frac{1}{\tau+1}+\frac{1}{\lambda-\tau-1}\Big). \nonumber
\end{eqnarray}

Case 2. When $\lambda-\tau-1>0$ and $\tau>0$, we have, for $x\in (0,1]$, 
\begin{eqnarray}\label{s-1}
S(x)&=&\sum_{m=1}^{\infty}\frac{m^{\tau}}{m^{\lambda}}=\sum_{m=1}^{\infty}\frac{1}{m^{\lambda-\tau}}\leq C_1 x^{\tau+1-\lambda}.\nonumber
\end{eqnarray}
Here we can take $C_1=\sum_{m=1}^{\infty}\frac{1}{m^{\lambda-\tau}}$ since $x^{\tau+1-\lambda}\geq 1$. 

When $\lambda-\tau-1>0$ and $\tau>0$, we have, for $x>1$, 
\begin{eqnarray}\label{s-2}
S(x)&=&\sum_{m=1}^{\lceil x \rceil}\frac{m^{\tau}}{x^{\lambda}}+\sum_{m=\lceil x \rceil+1}^{\infty}\frac{m^{\tau}}{m^{\lambda}}\nonumber \\
&\leq&  \lceil x \rceil \frac{\lceil x \rceil^{\tau}}{x^{\lambda}}+\int_{\lceil x \rceil}^{\infty}\frac{1}{t^{\lambda-\tau}}dt\nonumber \\
&\leq & x^{\tau+1-\lambda}+\frac{1}{\lambda-\tau-1}\lceil x \rceil^{\tau+1-\lambda}.\nonumber
\end{eqnarray}
Here, $\lceil x \rceil$ is the ceiling function. Note that $\lceil x \rceil^{\tau+1-\lambda}=1$ for each $x\in (1,2)$ so that
$$\lceil x \rceil^{\tau+1-\lambda}\leq 2^{\lambda-\tau-1}x^{\tau+1-\lambda},$$
 and for each $x\geq 2$, we have $\lceil x \rceil\geq \frac{1}{2}x$ so that
$$\lceil x \rceil^{\tau+1-\lambda} \leq \frac{1}{2^{\tau+1-\lambda}}x^{\tau+1-\lambda}.$$ 
Consequently, we obtain that, for $x>1$, 
\begin{eqnarray}
S(x)&\leq & C_2 x^{\tau+1-\lambda}. \nonumber 
\end{eqnarray}
Here we can take 
$C_2=2^{\lambda-\tau-1}+\frac{2^{\lambda-\tau-1}}{\lambda-\tau-1}.$
Hence, when $\lambda-\tau-1>0$ and $\tau>0$, it holds that, for any $x>0$, 
$$S(x)\leq (C_1+C_2)x^{\tau+1-\lambda}=C x^{\tau+1-\lambda}.$$
%Case 3. For $\lambda-\tau-1>0$ and $\tau\leq -1$, we have, for $x>1$, 
%\begin{eqnarray}\label{s-3} S(x)&=&\sum_{m=1}^{\lceil x \rceil}\frac{m^{\tau}}{x^{\lambda}}+\sum_{m=\lceil x \rceil+1}^{\infty}\frac{m^{\tau}}{m^{\lambda}}\nonumber \\&\leq& x^{-\lambda}(1+\int_{1}^{\lceil x \rceil}t^{\tau}dt)+\int_{\lceil x \rceil}^{\infty}\frac{1}{t^{\lambda-\tau}}dt.\nonumber\end{eqnarray}
%Then, when $\tau=-1$ so that $\lambda>0$, we easily see that there is a constant $C_2>0$, which is independent of $x$, such that 
% \begin{eqnarray}\label{s-4}
%S(x)&\leq& x^{-\lambda}(1+\ln\lceil x \rceil)+\frac{1}{\lambda}\lceil x \rceil^{-\lambda}\leq C_3 x^{-\lambda-\epsilon},\nonumber
%\end{eqnarray}
%for any $\epsilon>0$ since $\lim_{t\rightarrow +\infty}\frac{\ln t}{t^{\epsilon}}=0$ for any $\epsilon>0$. When $\tau<-1$ so that $x^{\tau+1}\leq 1$ for $x>1$, then we get that 
% \begin{eqnarray}\label{s-4}
%S(x)&\leq& x^{-\lambda}[1+\frac{1}{\tau+1}(\lceil x \rceil^{\tau+1}-1)]+\frac{1}{\lambda-\tau-1}\lceil x \rceil^{\tau+1-\lambda}\leq C_3 x^{-\lambda}.\nonumber
%\end{eqnarray}
Here, $C>0$ is independent of $x$. This proves Lemma \ref{ll-5}.
\end{proof}

\begin{remark}\label{f-re}When $\lambda-\tau-1>0$, for fixed $x>0$, we note that $\frac{m^{\tau}}{[\max\{m,x\}]^{\lambda}}$ is decreasing with respect to $m$ for $m\geq\lceil x \rceil+1$. 
Then it is not hard to see that, for fixed $x>0$,
$$  
S(x)\geq \int_{\lceil x \rceil+1}^{\infty}x^{\tau-\lambda}dx\geq \frac{(\lceil x \rceil+1)^{\tau+1-\lambda}}{\lambda-\tau-1}\geq \frac{(x+1)^{\tau+1-\lambda}}{\lambda-\tau-1}.$$
\end{remark}

\begin{lemma}\label{ll-5-1}Let $\beta, \tau, \lambda$ be real numbers. 
 If $\beta<\lambda$, $\tau=-1$, $\lambda>0$, or $\beta \leq \lambda$, $\tau<-1$, $\lambda-\tau-1>0$, then 
 $$A:=\sup_{x>1}x^{\beta}\sum_{m=1}^{\infty}\frac{m^{\tau}}{[\max\{m,x\}]^{\lambda}}<\infty.$$
\end{lemma}
\begin{proof}
We still let
$$S(x)=\sum_{m=1}^{\infty}\frac{m^{\tau}}{[\max\{m,x\}]^{\lambda}}.$$
When $\lambda-\tau-1>0$ and $\tau\leq -1$, we have, for $x>1$, 
\begin{eqnarray}\label{s-3}
S(x)&=&\sum_{m=1}^{\lceil x \rceil}\frac{m^{\tau}}{x^{\lambda}}+\sum_{m=\lceil x \rceil+1}^{\infty}\frac{m^{\tau}}{m^{\lambda}}\nonumber \\
&\leq& x^{-\lambda}(1+\int_{1}^{\lceil x \rceil}t^{\tau}dt)+\int_{\lceil x \rceil}^{\infty}\frac{1}{t^{\lambda-\tau}}dt.\nonumber 
\end{eqnarray}
We will consider the following two cases. 

Case 1. When $\tau=-1$, we easily see that there is a constant $C_1>0$, which is independent of $x$ and such that 
 \begin{eqnarray}\label{s-4}
S(x)&\leq& x^{-\lambda}(1+\ln\lceil x \rceil)+\frac{1}{\lambda}\lceil x \rceil^{-\lambda}\leq C_1 x^{-\lambda+\epsilon},\nonumber
\end{eqnarray}
for any $\epsilon>0$ since $\lim_{t\rightarrow +\infty}\frac{\ln t}{t^{\epsilon}}=0$ for any $\epsilon>0$. Consequently, by $\beta-\lambda<0$, we obtain that
$$A\leq C_1\sup_{x>1}x^{\beta}{x^{-\lambda+\epsilon}}=C_1\sup_{x>1}{x^{\beta-\lambda+\epsilon}}\leq C_1,$$
when $\varepsilon<\lambda-\beta$.

Case 2. When $\tau<-1$ so that $x^{\tau+1}\leq 1$ for $x>1$, then we get that 
 \begin{eqnarray}\label{s-4}
S(x)&\leq& x^{-\lambda}[1+\frac{1}{\tau+1}(\lceil x \rceil^{\tau+1}-1)]+\frac{1}{\lambda-\tau-1}\lceil x \rceil^{\tau+1-\lambda}\leq C_2 x^{-\lambda}.\nonumber
\end{eqnarray}
Here, $C_2>0$ is also independent of $x$. It follows from again $\beta-\lambda\leq 0$ that
$$A\leq C_2\sup_{x>1}x^{\beta}{x^{-\lambda}}=C_1\sup_{x>1}{x^{\beta-\lambda}}\leq C_2.$$
This finishes the proof of Lemma \ref{ll-5-1}.
\end{proof}
 
\begin{lemma}\label{ll-l}
Let $\lambda, \tau, \omega$ be real numbers. Then there is a constant $C>0$ such that
 $$\sup_{m\in \mathbb{N}}\frac{m^{\tau}}{[\max\{m,x\}]^{\lambda}}\leq C x^{\omega},$$
for a.e. $x\in \mathbb{R}_{+}$, if $\lambda\geq \tau-\omega$, $\lambda\geq \tau$, $\lambda\geq -\omega$, and $\omega \leq 0.$
\end{lemma}
\begin{proof}
It suffices to prove that, if $\lambda\geq \tau-\omega$, $\lambda\geq \tau$, $\lambda\geq -\omega$, and $\omega \leq 0$, then  \begin{equation}\label{good}\mathcal{S}:=\sup_{m\in \mathbb{N}, x\in \mathbb{R}_{+}}\frac{m^{\tau}x^{-\omega}}{[\max\{m,x\}]^{\lambda}}<\infty.\nonumber\end{equation}
First, we have
\begin{eqnarray}
\mathcal{S}_1:=\sup_{m\in \mathbb{N}, x\in (0,1]}\frac{m^{\tau}x^{-\omega}}{[\max\{m,x\}]^{\lambda}}=\sup_{m\in \mathbb{N}, x\in (0,1]}\frac{x^{-\omega}}{m^{\lambda-\tau}}\leq 1<\infty.\nonumber
\end{eqnarray}
When $x>1$, we have
\begin{eqnarray}
\mathcal{S}_2:=\sup_{m<x, m\in \mathbb{N}}\frac{m^{\tau}x^{-\omega}}{[\max\{m,x\}]^{\lambda}}=\sup_{m<x, m\in \mathbb{N}}\frac{m^{\tau}}{x^{\lambda+\omega}}.\nonumber
\end{eqnarray}
Case I. If $\tau>0$, then we have 
\begin{eqnarray}
\mathcal{S}_2\leq \sup_{m<x, m\in \mathbb{N}}\frac{1}{x^{\lambda+\omega-\tau}}\leq 1<\infty.\nonumber
\end{eqnarray}
Case II. If $\tau\leq 0$, then we have 
\begin{eqnarray}
\mathcal{S}_2\leq \sup_{m<x, m\in \mathbb{N}}\frac{1}{x^{\lambda+\omega}}\leq 1<\infty.\nonumber
\end{eqnarray}
On the other hand, we have
\begin{eqnarray}
\mathcal{S}_3:&=&\sup_{m\geq x, m\in \mathbb{N}}\frac{m^{\tau}x^{-\omega}}{[\max\{m,x\}]^{\lambda}}\nonumber \\
&=&\sup_{m\geq x, m\in \mathbb{N}}\frac{x^{-\omega}}{m^{\lambda-\tau}}\nonumber
\\&\leq& \sup_{m>x, m\in \mathbb{N}}\frac{1}{m^{\lambda+\omega-\tau}}<\infty.\nonumber
\end{eqnarray}
The lemma follows from $\mathcal{S}=\max\{\mathcal{S}_1, \mathcal{S}_2, \mathcal{S}_3\}$. The proof of Lemma \ref{ll-l} is complete.
\end{proof}

The following lemmas provide the necessary conditions for the boundedness of ${\bf H}_{\lambda, \alpha, \beta}$.
\begin{lemma}\label{ll-4-1}
Let $1\leq p, q<\infty$. Let $\lambda, \alpha, \beta, \varphi, \psi$ be real numbers and $\mathbf{H}_{\lambda, \alpha, \beta}$ be as in (\ref{ope}). We have

{\bf (1)} If $1\leq p\leq q<\infty$ and $\mathbf{H}_{\lambda, \alpha, \beta}$ is bounded from 
$l_{\varphi}^p$ to $L_{\psi}^q$, then 
\begin{equation} 
\begin{cases}
\lambda\geq \alpha+\beta+1+\frac{\psi+1}{q}-\frac{\varphi+1}{p}, \\
-q\beta<\psi+1<q(\lambda-\beta). \nonumber 
\end{cases}
\end{equation}

{\bf (2)} If $1\leq q<p<\infty$ and $\mathbf{H}_{\lambda, \alpha, \beta}$ is bounded from 
$l_{\varphi}^p$ to $L_{\psi}^q$, then 
\begin{equation} 
\begin{cases}
\lambda>\alpha+\beta+1+\frac{\psi+1}{q}-\frac{\varphi+1}{p}, \\
-q\beta<\psi+1<q(\lambda-\beta). \nonumber 
\end{cases}
\end{equation}
\end{lemma}
 
\begin{proof}
For  $1\leq p, q<\infty$, we suppose that $\mathbf{H}_{\lambda, \alpha, \beta}$ is bounded from $l_{\varphi}^{p}$ to $L_{\psi}^{q}$.  Take $a=\{a_m\}_{m=1}^{\infty}$ with $a_1=1$ and $a_m=0$ for $m\geq 2.$ Then we see that $\|a\|_{p,\varphi}=1,$ and 
 $$\mathbf{H}_{\lambda, \alpha, \beta}a(x)=\frac{x^{\beta}}{[\max\{1,x\}]^{\lambda}},$$
so that $$\|\mathbf{H}_{\lambda, \alpha, \beta}a\|_{q, \psi}^q=\int_{\mathbb{R}_{+}}\frac{x^{q\beta+\psi}}{[\max\{1,x\}]^{q\lambda}}dx<\infty.$$
This implies that $q\lambda-q\beta-\psi>1$ and $q\beta+\psi+1>0$. That is 
$$-q\beta<\psi+1<q(\lambda-\beta).$$
Now, we let 
$$\rho:=\alpha+\beta+1+\frac{\psi+1}{q}-\frac{\varphi+1}{p}-\lambda.$$  
For $\varepsilon>0$, we take $a_{\varepsilon}:=\{a_m^{\varepsilon}\}_{m=1}^{\infty}$ with
\begin{equation}\label{ff-2}a_m^{\varepsilon}=m^{-\frac{\varphi+1}{p}-\frac{\varepsilon}{p}}, m\in \mathbb{N}. 
\end{equation}
Then \begin{eqnarray}\frac{1}{\varepsilon}=\int_{1}^{\infty}x^{-1-\varepsilon}dx&\leq&\|a_{\varepsilon}\|_{p, \varphi}^p=\sum_{n=1}^{\infty}n^{-1-\varepsilon}\nonumber \\
&<&1+\int_{1}^{\infty}x^{-1-\varepsilon}dx=1+\frac{1}{\varepsilon}.\nonumber\end{eqnarray}
It follows from the boundedness of ${\bf H}_{\lambda, \alpha, \beta}$ that, for a.e. $x\in \mathbb{R}_{+}$,  
$$\mathbf{H}_{\lambda, \alpha, \beta}a_{\varepsilon}(x)=\sum_{m=1}^{\infty}\frac{m^{\alpha}x^{\beta}}{[\max\{m,x\}]^{\lambda}}m^{-\frac{\varphi+1}{p}-\frac{\varepsilon}{p}}<\infty.$$
Consequently, from Lemma \ref{ll-5}, we know that 
\begin{eqnarray}\lambda-\alpha+\frac{\varphi+1}{p}+\frac{\varepsilon}{p}>1,\nonumber \end{eqnarray}
for any $\varepsilon>0$. Hence, in view of Remark \ref{f-re}, we have  
\begin{eqnarray}\label{h-2}\|{\bf H}_{\lambda, \alpha, \beta}a_{\varepsilon}\|_{q,\psi}^q&=&\int_{0}^{\infty}x^{\psi}\left [\sum_{m=1}^{\infty}\frac{m^{\alpha-\frac{\varphi+1}{p}-\frac{\varepsilon}{p}}x^{\beta}}{[\max\{m,x\}]^{\lambda}}\right ]^q dx \nonumber \\ 
&\geq & C_1^{-q}\int_{1}^{\infty}x^{\psi}\left[x^{\beta}(x+1)^{\alpha-\frac{\varphi+1}{p}-\frac{\varepsilon}{p}+1-\lambda}\right]^qdx \nonumber \\ 
&\geq & C_1^{-q}\int_{1}^{\infty}x^{\psi}\left[x^{\beta}(2x)^{\alpha-\frac{\varphi+1}{p}-\frac{\varepsilon}{p}+1-\lambda}\right]^qdx \nonumber \\
&=&  C_1^{-q}C_2^{q}\int_{1}^{\infty}x^{\psi+q(\alpha+\beta+1-\frac{\varphi+1}{p}-\frac{\varepsilon}{p}-\lambda)}dx  \nonumber \\
&=& C_1^{-q}C_2^{q}\int_{1}^{\infty}x^{-1+q\rho-\frac{q\varepsilon}{p}}dx. \nonumber
\end{eqnarray}
Here \begin{equation}\label{cc}C_1=\lambda-\alpha+\frac{\varphi+1}{p}+\frac{\varepsilon}{p}-1, \,\,\, C_2=2^{-C_1}.\end{equation}
Thus, if $\rho>0$, when $\varepsilon<q\rho$, we obtain that $$\int_{1}^{\infty}x^{-1+q\rho-\frac{q\varepsilon}{p}}dx=\infty,$$
so that $\|\mathbf{H}_{\lambda, \alpha, \beta}a_{\varepsilon}\|_{q,\psi}^q=\infty$, which contradicts the boundedness of $\mathbf{H}_{\lambda, \alpha, \beta}$. This implies that $\rho\leq 0$. That is
$$\lambda\geq \alpha+\beta+1+\frac{\psi+1}{q}-\frac{\varphi+1}{p}.$$
To finish the proof, we next only need to show that, when $1\leq q<p<\infty$, if \begin{equation}\label{add-l} 
\begin{cases}
\lambda=\alpha+\beta+1+\frac{\psi+1}{q}-\frac{\varphi+1}{p}, \\
-q\beta<\psi+1<q(\lambda-\beta),\\ 
\end{cases}
\end{equation}  
then ${\bf H}_{\lambda, \alpha, \beta}$ is not bounded from $l_{\varphi}^p$ to $L_{\psi}^q$. We suppose that $1\leq q<p<\infty$ and (\ref{add-l}) holds. For $\varepsilon>0$, we take $a_{\varepsilon}=\{a_m^{\varepsilon}\}_{m=1}^{\infty}$ as above, then $\|a\|_{p, \varphi}^p=\frac{1}{\varepsilon}[1+o(1)]$ as $\varepsilon \rightarrow 0$ and by repeating some above arguments, we have
\begin{eqnarray}\label{h-3}\|{\bf H}_{\lambda, \alpha, \beta}a_{\varepsilon}\|_{q,\psi}^q&\geq&  C_1^{-q}C_2^{q}\int_{1}^{\infty}x^{\psi+q(\alpha+\beta+1-\frac{\varphi+1}{p}-\frac{\varepsilon}{p}-\lambda)}dx  \nonumber \\
&=& C_1^{-q}C_2^{q}\int_{1}^{\infty}x^{-1-\frac{q\varepsilon}{p}}dx=C_1^{-q}C_2^{q}\frac{p}{q\varepsilon}. \nonumber
\end{eqnarray}
Here $C_1$ and $C_2$ are the same as in (\ref{cc}). It follows that
\begin{eqnarray}
\frac{\|{\bf H}_{\lambda, \alpha, \beta}a_{\varepsilon}\|_{q,\psi}}{\|a\|_{p, \varphi}}\geq C_1^{-1}C_2(pq^{-1})^{\frac{1}{q}}\varepsilon^{\frac{1}{p}-\frac{1}{q}}\rightarrow \infty, \nonumber 
\end{eqnarray}
as $\varepsilon \rightarrow 0$ since $p>q$. This means that ${\bf H}_{\lambda, \alpha, \beta}$ is not bounded from $l_{\varphi}^p$ to $L_{\psi}^q$. The proof of Lemma \ref{ll-4-1} is finished.
\end{proof}

\begin{lemma}\label{ll-5-add}
Let $\lambda, \alpha, \beta, \varphi, \psi$ be real numbers and $\mathbf{H}_{\lambda, \alpha, \beta}$ be as in (\ref{ope}). 

{\bf (1)}If $\mathbf{H}_{\lambda, \alpha, \beta}$ is bounded from 
$l_{\varphi}^1$ to $L^{\infty}$, then  
\begin{equation} 
\begin{cases}
\lambda\geq \beta\geq 0, \\
\lambda\geq \alpha+\beta-\varphi. \nonumber 
\end{cases}
\end{equation}

{\bf (2)}If $1<p<\infty$ and $\mathbf{H}_{\lambda, \alpha, \beta}$ is bounded from 
$l_{\varphi}^p$ to $L^{\infty}$, then  
\begin{equation}\label{ad-eq-1}
\begin{cases}
\lambda>\beta\geq0, \\
\lambda\geq\alpha+\beta+1-\frac{\varphi+1}{p}, 
\end{cases}
\end{equation}
or
\begin{equation}\label{ad-eq-2}
\begin{cases}
\lambda\geq \beta\geq0, \\
\lambda>\alpha+\beta+1-\frac{\varphi+1}{p}. 
\end{cases}
\end{equation}
\end{lemma}
\begin{proof}
For $1\leq p<\infty$, we suppose that  $\mathbf{H}_{\lambda, \alpha, \beta}$ is bounded from $l_{\varphi}^p$ to $L^{\infty}$. Take $a=\{a_m\}_{m=1}^{\infty}$ with $a_1=1$ and $a_m=0$ for $m\geq 2$, then we have
$\|a\|_{p,\alpha}=1$ and
\begin{equation}
\mathbf{H}_{\lambda, \alpha, \beta}a(x)=\frac{x^{\beta}}{[\max\{1,x\}]^{\lambda}},\nonumber 
\end{equation}   
so that \begin{equation}
\|\mathbf{H}_{\lambda, \alpha, \beta}a\|_{\infty}=\sup_{x\in \mathbb{R}_{+}}\frac{x^{\beta}}{[\max\{1,x\}]^{\lambda}}<\infty.\nonumber 
\end{equation} 
Consequently, when $x\in (0,1]$, we have $$\sup_{x\in (0,1]}\frac{x^{\beta}}{[\max\{1,x\}]^{\lambda}}=\sup_{x\in (0,1]}x^{\beta}<\infty.$$ This implies that $\beta\geq 0.$  When $x>1$, we have 
$$\sup_{x>1}\frac{x^{\beta}}{[\max\{1,x\}]^{\lambda}}=\sup_{x>1}x^{\beta-\lambda}<\infty,$$
This implies that $\lambda\geq \beta.$

Now, for $\varepsilon>0$, we take the same $a_{\varepsilon}=\{a_m^{\varepsilon}\}_{m=1}^{\infty}$ as in (\ref{ff-2}). Then we see that
$\|a_{\varepsilon}\|_{p, \varphi}^p=\frac{1}{\varepsilon}[1+o(1)]$ as $\varepsilon \rightarrow 0$, and 
 \begin{equation}
\mathbf{H}_{\lambda, \alpha, \beta}a_{\varepsilon}(x)=\sum_{m=1}^{\infty}\frac{m^{\alpha-\frac{\varphi+1}{p}-\frac{\varepsilon}{p}}x^{\beta}}{[\max\{m,x\}]^{\lambda}}<\infty,\nonumber 
\end{equation}  
for a.e. $x\in \mathbb{R}_{+}$. By Lemma \ref{ll-5}, we obtain that
$$\lambda-\alpha+\frac{\varphi+1}{p}+\frac{\varepsilon}{p}-1>0,$$
for any $\varepsilon>0$. %This implies that $$\lambda\geq \alpha+1-\frac{\varphi+1}{p}.$$
Furthermore, using again Lemma \ref{ll-5}, we see that, for each $x>0$, 
 \begin{equation}
\mathbf{H}_{\lambda, \alpha, \beta}a_{\varepsilon}(x)\geq \frac{x^{\beta}(x+1)^{\alpha-\frac{\varphi+1}{p}-\frac{\varepsilon}{p}+1-\lambda}}{\lambda-\alpha+\frac{\varphi+1}{p}+\frac{\varepsilon}{p}-1}:=C_1x^{\beta}(x+1)^{\alpha-\frac{\varphi+1}{p}-\frac{\varepsilon}{p}+1-\lambda}.\nonumber 
\end{equation} 
Hence we get that
$$\infty>\|\mathbf{H}_{\lambda, \alpha, \beta}a_{\varepsilon}\|_{\infty} \geq C \sup_{x\in \mathbb{R}_{+}}x^{\beta}(x+1)^{\alpha-\frac{\varphi+1}{p}-\frac{\varepsilon}{p}+1-\lambda}.$$
It follows that 
$$\infty>\sup_{x>1} x^{\beta}(x+1)^{\alpha-\frac{\varphi+1}{p}-\frac{\varepsilon}{p}+1-\lambda}\geq \sup_{x>1}x^{\alpha+\beta-\frac{\varphi+1}{p}-\frac{\varepsilon}{p}+1-\lambda}.$$
This implies that
$$\alpha+\beta-\frac{\varphi+1}{p}-\frac{\varepsilon}{p}+1-\lambda \leq 0,$$
for any $\varepsilon>0$ so that
\begin{equation}\label{ad-eq-3}\lambda\geq \alpha+\beta+1-\frac{\varphi+1}{p}.\end{equation}
From the above arguments, we have proved that the boundedness of $\mathbf{H}_{\lambda, \alpha, \beta}: l_{\varphi}^p \rightarrow L^{\infty}$ implies that  
\begin{equation} 
\begin{cases}
\lambda\geq \beta\geq 0, \\
\lambda\geq \alpha+\beta+1-\frac{\varphi+1}{p}. \nonumber 
\end{cases}
\end{equation}
This means that the part {\bf (1)}, the calse $p=1$, of Lemma \ref{ll-5-add} is true. 

Next, we prove the part {\bf (2)}, the case $1<p<\infty$, of Lemma \ref{ll-5-add}.  
To finish the proof, it is enough to show that $\mathbf{H}_{\lambda, \alpha, \beta}: l_{\varphi}^p \rightarrow L^{\infty}$ is not bounded if $\lambda=\alpha+\beta+1-\frac{\varphi+1}{p}$ and $\lambda=\beta\geq 0$. Now, we assume that
$\lambda=\alpha+\beta+1-\frac{\varphi+1}{p}$ and $\lambda=\beta\geq 0$, it follows that $\alpha+1-\frac{\varphi+1}{p}=0$.  Take $a_{\varepsilon}$ as above, then we have 
$$\|a_{\varepsilon}\|_{p, \varphi}=\frac{1}{{\varepsilon}^{\frac{1}{p}}}[1+o(1)],\,\, \varepsilon \rightarrow 0,$$ 
and for a natural number $x$,
 \begin{eqnarray}
\mathbf{H}_{\lambda, \alpha, \beta}a_{\varepsilon}(x)&=&\sum_{m=1}^{\infty}\frac{x^{\beta}m^{\alpha-\frac{\varphi+1}{p}-\frac{\varepsilon}{p}}}{[\max\{m,x\}]^{\lambda}}
\nonumber \\
&=&\sum_{m=1}^{x}m^{\alpha-\frac{\varphi+1}{p}-\frac{\varepsilon}{p}}+x^{\beta}\sum_{m=x+1}^{\infty}m^{-1-\beta-\frac{\varepsilon}{p}}
\nonumber \\
&\geq & \int_{1}^{x+1}t^{-1-\frac{\varepsilon}{p}}dx+x^{\beta}\int_{x+1}^{\infty}t^{-1-\beta-\frac{\varepsilon}{p}}dx\nonumber \\
&\geq & \frac{p}{\varepsilon}(1-(x+1)^{-\frac{\varepsilon}{p}})+\frac{px^{\beta}}{p\beta+\varepsilon}(x+1)^{-\beta-\varepsilon/p}.\nonumber 
\end{eqnarray} 
We can check that
$$\frac{p}{\varepsilon}(1-(x+1)^{-\frac{\varepsilon}{p}})+\frac{px^{\beta}}{p\beta+\varepsilon}(x+1)^{-\beta-\varepsilon/p}\to \frac{p}{\varepsilon},$$
as $x\to \infty$.  
It follows that $\|\mathbf{H}_{\lambda, \alpha, \beta}a_{\varepsilon}\|_{\infty}\geq \frac{p}{2\varepsilon}$ so that
\[\frac{\|\mathbf{H}_{\lambda, \alpha, \beta}a_{\varepsilon}\|_{\infty}}{\|a_{\varepsilon}\|_{p,\varphi}}\geq \frac{p}{2{\varepsilon}^{1/p'}}[1+o(1)]\to \infty,\,{\text as}\,\,\varepsilon \rightarrow 0.\]
This means that $\mathbf{H}_{\lambda, \alpha, \beta}: l_{\varphi}^p \rightarrow L^{\infty}$ is not bounded when $\lambda=\alpha+\beta+1-\frac{\varphi+1}{p}$ and $\lambda=\beta\geq 0$. This proves the part {\bf (2)} of Lemma \ref{ll-5-add}. The proof of Lemma \ref{ll-5-add} is complete. 
\end{proof}

\section{{\bf Proof of Theorem \ref{m-th-1}}}
\subsection{{\bf Proof of Theorem \ref{m-th-1} }}Note that the necessity of the boundedness of ${\bf H}_{\lambda, \alpha, \beta}$ has been proved by Lemma \ref{ll-4-1}. We only need to prove that, for $1\leq p\leq q<\infty$, if 
\begin{equation}\label{e-1}\lambda\geq \alpha+\beta+1+\frac{\psi+1}{q}-\frac{\varphi+1}{p},\end{equation}
and
\begin{equation}\label{e-2}-q\beta<\psi+1<q(\lambda-\beta).\end{equation}
Then ${\bf H}_{\lambda, \alpha, \beta}$ is bounded $l_{\varphi}^{p}$ to $L_{\psi}^{q}$. First, from (\ref{e-2}), we see that  
\begin{equation}\label{e-3} \lambda> \frac{\psi+1}{q}+\beta >0.\end{equation}
We will divide our proof into the following two cases.
For the sake of simplicity, we will write
$$k(m,x)=\frac{m^{\alpha}x^{\beta}}{[\max\{m,x\}]^{\lambda}}.$$
{\bf Case I. $1<p\leq q<\infty.$} As $\psi+1<q(\lambda-\beta)$, i.e. $q(\lambda-\beta)-\psi-1>0$, then we can find a constant $t>1$ such that 
\begin{equation}\label{add-0}q(\lambda-\beta)-\psi-1>\frac{q\lambda}{t}.\end{equation} 
Let $s>1$ be such that $\frac{1}{s}+\frac{1}{t}=1$.  It follows that
\begin{equation}\label{add-1}-\frac{\psi+1}{q}-\beta+\frac{\lambda}{s}>0.\end{equation} From the left part of (\ref{e-2}), we have
$-\frac{\psi+1}{q}-\beta<0$ so that 
\begin{equation}\label{add-2}L_1:=-\frac{\psi+1}{q}-\beta-\frac{1}{p'}-\frac{\alpha}{s}+\frac{\lambda}{s}<-\frac{1}{p'}-\frac{\alpha}{s}+\frac{\lambda}{s}:=R_1.\end{equation}
Also, we see from (\ref{add-1}), (\ref{add-0}) and $\lambda>0$ that 
\begin{equation}\label{add-3}L_1<L_2:=-\frac{1}{p'}-\frac{\alpha}{s}<-\frac{1}{p'}-\frac{\alpha}{s}+\frac{\lambda}{s}=R_1,\end{equation}
and 
\begin{equation}\label{add-4} L_2<-\frac{\psi+1}{q}-\beta-\frac{1}{p'}-\frac{\alpha}{s}+\lambda:=R_2.\end{equation}
Consequently, it follows from (\ref{add-2}), (\ref{add-3}) and (\ref{add-4}) that $(L_1, R_1)\cup (L_2, R_2)$ so that we can take a constant $A$ such that 
\begin{equation}\label{e-4} 
-\frac{1}{p'}-\frac{\alpha}{s}<A<-\frac{1}{p'}-\frac{\alpha}{s}+\frac{\lambda}{s},\end{equation}
and 
\begin{equation}\label{e-5}-\frac{\psi+1}{q}-\beta-\frac{\alpha}{s}-\frac{1}{p'}+\frac{\lambda}{s}<A<-\frac{\psi+1}{q}-\beta-\frac{\alpha}{s}-\frac{1}{p'}+\lambda.\end{equation}
It is easy to see that (\ref{e-4}) is equivalent to
\begin{equation}\label{equi-1}
\begin{cases}
\frac{\alpha}{s}p'+p'A>-1, \\
\frac{\lambda}{s}p'-\frac{\alpha}{s}p'-p'A-1>0,
\end{cases}
\end{equation}
and (\ref{e-5}) is equivalent to
\begin{equation}\label{equi-2}
\begin{cases}
\frac{\beta}{t}q+B+\psi>-1, \\
\frac{\lambda}{t}p-\frac{\beta}{t}p-B-\psi-1>0.
\end{cases}
\end{equation}
Here 
$$B=q[\frac{\beta}{s}+\frac{\alpha}{s}+A+\frac{1}{p'}-\frac{\lambda}{s}].$$
Now, for $a=\{a_m\}_{m=1}^{\infty}\in l_{\varphi}^{p}$ with $a_m\geq 0$ for any $m\in \mathbb{N}$, we have, for $x\in \mathbb{R}_{+}$, 
\begin{eqnarray}
{\bf H}_{\lambda, \alpha, \beta}a(x)&=&\sum_{m=1}^{\infty}k(m,x)a_m\nonumber \\
&=& \sum_{m=1}^{\infty}[k(m,x)]^{\frac{1}{s}}m^{A}\cdot[k(m,x)]^{\frac{1}{t}}m^{-A}a_m. \nonumber
\end{eqnarray}
By H\"{o}lder's inequality, we obtain that
\begin{eqnarray}
{\bf H}_{\lambda, \alpha, \beta}a(x)&\leq & \Big\{\sum_{m=1}^{\infty}[k(m,x)]^{\frac{p'}{s}}m^{p'A}\Big\}^{\frac{1}{p'}}\Big\{[k(m,x)]^{\frac{p}{t}}m^{-pA}a_m^p\Big\}^{\frac{1}{p}} \nonumber \\
&=& \Big\{\sum_{m=1}^{\infty}\frac{m^{\frac{\alpha}{s}p'+p'A}x^{\frac{\beta}{s}p'}}{[\max\{m,x\}]^{\frac{\lambda}{s}p'}}\Big\}^{\frac{1}{p'}}\Big\{\sum_{m=1}^{\infty}[k(m,x)]^{\frac{p}{t}}m^{-pA}a_m^p\Big\}^{\frac{1}{p}}
\nonumber \\
&:=& {\bf S}_1^{^{\frac{1}{p'}}}{\bf S}_2^{^{\frac{1}{p}}}.\nonumber
\end{eqnarray}
By Lemma \ref{ll-5} and (\ref{equi-1}), we see that there is a constant ${\bf C}_1>0$ such that
\begin{eqnarray}
{\bf S}_1\leq {\bf C}_1 x^{\frac{\alpha}{s}p'+\frac{\beta}{s}p'+p'A+1-\frac{\lambda}{s}p'}.\nonumber
\end{eqnarray}
Note that we can take 
\begin{eqnarray}
{\bf C}_1=\frac{1}{\frac{\alpha}{s}p'+p'A+1}+\frac{1}{\frac{\lambda}{s}p'-\frac{\alpha}{s}p'-p'A-1}, \nonumber
\end{eqnarray}
if \begin{equation}\label{con-1}\frac{\alpha}{s}p'+p'A\leq 0.\end{equation}
Consequently, we obtain that
 \begin{eqnarray}
\|{\bf H}_{\lambda, \alpha, \beta}a\|_{q, \psi}&\leq & \Big[\int_{0}^{\infty}{\bf S}_1^{^{\frac{q}{p'}}}{\bf S}_2^{^{\frac{q}{p}}}x^{\psi}dx\Big]^{\frac{1}{q}}\nonumber \\
&=&{\bf C}_1^{\frac{1}{p'}}\Big[\int_{0}^{\infty}x^{q(\frac{\alpha}{s}+\frac{\beta}{s}+A+\frac{1}{p'}-\frac{\lambda}{s})}\nonumber \\
&&\quad\quad\quad \times \Big\{\sum_{n=1}^{\infty}[k(m,x)]^{\frac{p}{t}}m^{-pA}a_m^p\Big\}^{\frac{q}{p}}x^{\psi}dx\Big]^{\frac{1}{q}}
\nonumber \\
&=& {\bf C}_1^{\frac{1}{p'}}\Big[\int_{0}^{\infty}x^{B+\psi}\Big\{\sum_{m=1}^{\infty}[k(m,x)]^{\frac{p}{t}}m^{-pA}a_m^p\Big\}^{\frac{q}{p}}dx\Big]^{\frac{1}{q}}. \nonumber
\end{eqnarray}
Then it follows from Minkowski's inequality that 
 \begin{eqnarray}
\|{\bf H}_{\lambda, \alpha, \beta}a\|_{q, \psi}&\leq & {\bf C}_1^{\frac{1}{p'}}\Big[\int_{0}^{\infty}x^{B+\psi}\Big\{\sum_{m=1}^{\infty}[k(m,x)]^{\frac{p}{t}}m^{-pA}a_m^p\Big\}^{\frac{q}{p}}dx\Big]^{\frac{p}{q}\cdot\frac{1}{p}} \nonumber \\
&\leq & {\bf C}_1^{\frac{1}{p'}}\Big[\sum_{m=1}^{\infty}\Big\{\int_{0}^{\infty}[k(m,x)]^{\frac{q}{t}}x^{B+\psi}dx\Big\}^{\frac{p}{q}}m^{-pA}a_m^p\Big]^{\frac{1}{p}} \nonumber \\
&=& {\bf C}_1^{\frac{1}{p'}}\Big[\sum_{m=1}^{\infty}\Big\{\int_{0}^{\infty}\frac{m^{\frac{\alpha}{t}q}x^{\frac{\beta}{t}q+B+\psi}}{[\max\{m,x\}]^{\frac{\lambda}{t}q}}dx\Big\}^{\frac{p}{q}}m^{-pA}a_m^p\Big]^{\frac{1}{p}}\nonumber \\
&:=& {\bf C}_1^{\frac{1}{p'}}\Big[\sum_{m=1}^{\infty}{\bf I}_1^{\frac{p}{q}}m^{-pA}a_m^p\Big]^{\frac{1}{p}}. \nonumber
\end{eqnarray}
On the other hand, by using (\ref{equi-2}) and Lemma \ref{ll-4}, we know that   
\begin{eqnarray}
{\bf I}_1={\bf C}_2 m^{\frac{\beta}{t}q+B+\psi+1-\frac{\lambda}{t}q+\frac{\alpha}{t}q}. \nonumber
\end{eqnarray}
Here
\begin{eqnarray}
{\bf C}_2=\frac{1}{\frac{\beta}{t}q+B+\psi+1}+\frac{1}{\frac{\lambda}{t}q-\frac{\beta}{t}q-B-\psi-1}.\nonumber
\end{eqnarray}
It follows that
 \begin{eqnarray}
\|{\bf H}_{\lambda, \alpha, \beta}a\|_{q, \psi}&\leq &{\bf C}_1^{\frac{1}{p'}}{\bf C}_2^{\frac{1}{q}}\Big[\sum_{m=1}^{\infty}m^{\frac{p}{q}[\frac{\beta}{t}q+B+\psi+1-\frac{\lambda}{t}q+\frac{\alpha}{t}q]}m^{-pA}a_m^p\Big]^{\frac{1}{p}}. \nonumber
\end{eqnarray}
Meanwhile, we note that
\begin{eqnarray}
\lefteqn{\frac{p}{q}[\frac{\beta}{t}q+B+\psi+1-\frac{\lambda}{t}q+\frac{\alpha}{t}q]-pA}\nonumber \\
&&=p[\frac{\beta}{t}+\frac{B}{q}+\frac{\psi+1}{q}-\frac{\lambda}{t}-A+\frac{\alpha}{t}]\nonumber \\
&&=p[\frac{\beta}{t}+(\frac{\alpha+\beta}{s}+A+\frac{1}{p'}-\frac{\lambda}{s})+\frac{\psi+1}{q}-\frac{\lambda}{t}-A+\frac{\alpha}{t}] \nonumber \\
&&=p[\beta+\alpha+\frac{1}{p'}-\lambda+\frac{\psi+1}{q}]\leq p[\frac{1}{p'}-1+\frac{\varphi+1}{p}]=\varphi. \nonumber
\end{eqnarray}
This means that  \begin{eqnarray}
\|{\bf H}_{\lambda, \alpha, \beta}a\|_{q, \psi}&\leq &{\bf C}_1^{\frac{1}{p'}}{\bf C}_2^{\frac{1}{q}}\|a\|_{p, \varphi}, \nonumber
\end{eqnarray}
so that ${\bf H}_{\lambda, \alpha, \beta}$ is bounded from $l_{\varphi}^{p}$ to $L_{\psi}^{q}$ in this case.

{\bf Case II. $1=p\leq q<\infty$}
Recall that the following conditions are satisfied. 
\begin{equation}\label{ee-2}\lambda>0,\,\, -q\beta<\psi+1<q(\lambda-\beta).\nonumber\end{equation}
From $-q\beta<\psi+1$, we know that there is a constant $s>1$ such that 
\begin{equation}\label{ine-1}\psi+1>-q\beta+\frac{\lambda}{s}.\end{equation}
Let $t>1$ be such that $\frac{1}{s}+\frac{1}{t}=1$. Note that $\psi+1<q(\lambda-\beta)$ so that 
\begin{equation}\label{ine-2}
\frac{\beta-\lambda}{s}<\frac{\lambda-\beta}{t}-\frac{\psi+1}{q},
\end{equation}
and from (\ref{ine-1}) we have  \begin{equation}\label{ine-3}
-\frac{\beta}{t}-\frac{\psi+1}{q}<\frac{\beta-\lambda}{s}<\frac{\beta}{s}.
\end{equation}
Then we see from (\ref{ine-2}) and (\ref{ine-3}) that we can take a constant $D$ such that
\begin{equation}\label{ine-4}
-\frac{\beta}{t}-\frac{\psi+1}{q}<D<\frac{\lambda-\beta}{t}-\frac{\psi+1}{q},
\end{equation}
and 
\begin{equation}\label{ine-5}
\frac{\beta-\lambda}{s}<D<\frac{\beta}{s}.
\end{equation}
Consequently, it is easy to see that (\ref{ine-4}) is equivalent to 
\begin{equation}\label{equi-3}
\begin{cases}
q\frac{\beta}{t}+qD+\psi+1>0, \\
q\frac{\lambda-\beta}{t}-qD-\psi-1>0, 
\end{cases}
\end{equation}
and  (\ref{ine-5}) is equivalent to  \begin{equation}\label{equi-4}
\begin{cases}
\frac{\lambda-\beta}{s}+D>0, \\
\frac{\beta}{s}-D>0. 
\end{cases}
\end{equation}
Hence we see from (\ref{equi-4}) and Lemma \ref{ll-l} that there is a constant ${\bf C}_3>0$ such that 
\begin{equation}
\sup_{m\in \mathbb{N}}\frac{m^{\frac{\lambda-\beta}{s}+D}}{[\max\{m,x\}]^{\frac{\lambda}{s}}} \leq {\bf C}_3 x^{D-\frac{\beta}{s}}, \nonumber 
\end{equation}
for all $x\in \mathbb{R}_{+}$. It follows that, for $a=\{a_m\}_{m=1}^{\infty}\in l_{\varphi}^{1}$ with $a_m\geq 0$ for any $m\in \mathbb{N}$, and $x\in \mathbb{R}_{+}$,
\begin{eqnarray}
{\bf H}_{\lambda, \alpha, \beta}a(x)&=&\sum_{m=1}^{\infty}[k(m,x)]^\frac{1}{s}[k(m,x)]^{\frac{1}{t}}a_m \nonumber \\
&=&\sum_{m=1}^{\infty}\frac{m^{\frac{\lambda-\beta}{s}+D}}{[\max\{m,x\}]^{\frac{\lambda}{s}}}m^{\frac{\alpha+\beta-\lambda}{s}-D}x^{\frac{\beta}{s}}[k(m,x)]^{\frac{1}{t}}a_m \nonumber \\
&\leq & {\bf C}_3 x^{D}\sum_{m=1}^{\infty}m^{\frac{\alpha+\beta-\lambda}{s}-D}[k(m,x)]^{\frac{1}{t}}a_m. \nonumber 
\end{eqnarray}
Consequently, we have
\begin{equation}
\|{\bf H}_{\lambda, \alpha, \beta}a\|_{q, \psi}\leq {\bf C}_3\Big[\int_{0}^{\infty}\Big(\sum_{m=1}^{\infty}m^{\frac{\alpha+\beta-\lambda}{s}-D}[k(m,x)]^{\frac{1}{t}}a_m\Big)^q x^{qD+\psi}dx\Big]^{\frac{1}{q}}. \nonumber
\end{equation}
By using Minkowski's inequality again, we obtain that 
\begin{eqnarray}
\|{\bf H}_{\lambda, \alpha, \beta}a\|_{q, \psi}&\leq& {\bf C}_3\sum_{m=1}^{\infty}\Big(\int_{0}^{\infty}[k(m,x)]^{\frac{q}{t}}x^{qD+\psi}dx\Big)^\frac{1}{q}m^{\frac{\alpha+\beta-\lambda}{s}-D}a_m \nonumber \\
&=& {\bf C}_3\sum_{m=1}^{\infty}\Big(\int_{0}^{\infty}\frac{m^{\frac{\alpha}{t}q}x^{\frac{\beta}{t}q+qD+\psi}}{[\max\{m,x\}]^{\frac{\lambda}{t}q}}dx\Big)^\frac{1}{q}m^{\frac{\alpha+\beta-\lambda}{s}-D}a_m\nonumber \\
&=&  {\bf C}_3\sum_{m=1}^{\infty}{{\bf I}_2}^\frac{1}{q}m^{\frac{\alpha+\beta-\lambda}{s}-D}a_m. \nonumber
\end{eqnarray}
From (\ref{equi-3}) and Lemma \ref{ll-5}, we see that there is a constant ${\bf C}_4>0$ such that 
\begin{eqnarray}
{{\bf I}_2}={\bf C}_4 m^{\frac{\alpha}{t}q+\frac{\beta}{t}q+qD+\psi+1-\frac{\lambda}{t}q}. \nonumber 
\end{eqnarray}
It follows that
\begin{eqnarray}
\|{\bf H}_{\lambda, \alpha, \beta}a\|_{q, \psi}&\leq& {\bf C}_3{\bf C}_4^{\frac{1}{q}}\sum_{m=1}^{\infty}m^{\frac{1}{q}[\frac{\alpha}{t}q+\frac{\beta}{t}q+qD+\psi+1-\frac{\lambda}{t}q]+\frac{\alpha+\beta-\lambda}{s}-D}a_m. \nonumber
\end{eqnarray}
Meanwhile, we note from $\lambda \geq \alpha+\beta+\frac{\psi+1}{q}-\varphi$ that
\begin{eqnarray}
\lefteqn{\frac{1}{q}[\frac{\alpha}{t}q+\frac{\beta}{t}q+qD+\psi+1-\frac{\lambda}{t}q]+\frac{\alpha+\beta-\lambda}{s}-D}\nonumber \\
&&=\alpha+\beta-\lambda+\frac{\psi+1}{q}\leq \varphi.  \nonumber
\end{eqnarray}
This means that \begin{eqnarray}
\|{\bf H}_{\lambda, \alpha, \beta}a\|_{q, \psi}&\leq& {\bf C}_3{\bf C}_4^{\frac{1}{q}}\|a\|_{1, \varphi}, \nonumber
\end{eqnarray}
so that ${\bf H}_{\lambda, \alpha, \beta}$ is bounded from $l_{\varphi}^{p}$ to $L_{\psi}^{q}$ for $1=p\leq q<\infty$. This finishes the proof of Theorem \ref{m-th-1}.

\section{{\bf Proof of Theorem \ref{m-th-2} and \ref{m-th-3}}}
\subsection{\bf Proof of Theorem \ref{m-th-2} }First, we note that the only if part has been proved by {\bf (1)} of Lemma \ref{ll-5-add}. We only need to prove the if part. We see that  ${\bf H}_{\lambda, \alpha, \beta}$ is bounded from $l_{\varphi}^1$ to $L^{\infty}$ is equivalent to the operator 
$$\widehat{{\bf H}}_{\lambda, \alpha, \beta}a(x):=\sum_{m=1}^{\infty} \frac{m^{\alpha-\varphi}x^{\beta}}{[\max\{m,x\}]^{\lambda}}a_m,\,\, a=\{a_m\}_{m=1}^{\infty}, \,\, n\in \mathbb{N},$$
is bounded from $l^1$ to $L^{\infty}$. Meanwhile, from \cite[Proposition 5.4]{Tao}, we know that $\widehat{{\bf H}}_{\lambda, \alpha, \beta}: l^1\rightarrow L^{\infty}$ is bounded if and only if  
\begin{equation}\label{l-l}\sup_{m\in \mathbb{N},x\in \mathbb{R}_{+}}\frac{m^{\alpha-\varphi}n^{\beta}}{[\max\{m,x\}]^{\lambda}}<\infty.\end{equation}
It is easy to see from Lemma \ref{ll-l} that, if \begin{equation}\label{th-2-1}
\begin{cases}
\lambda\geq \beta\geq 0, \\
\lambda\geq \alpha+\beta-\varphi.\nonumber
\end{cases}
\end{equation}
then (\ref{l-l}) holds so that $\widehat{{\bf H}}_{\lambda, \alpha, \beta}: l^1\rightarrow l^{\infty}$ is bounded. This proves the if part of Theorem \ref{m-th-2} and the proof of Theorem \ref{m-th-2} is finished.

\subsection{{\bf Proof of Theorem \ref{m-th-3}}}
We first note that the only if part has been proved by {\bf (2)} of Lemma \ref{ll-5-add}.  We only need to prove that, for $1<p<\infty$ , if 
\begin{equation}\label{ad-eq-1-1}
\begin{cases}
\lambda>\beta\geq 0, \\
\lambda\geq \alpha+\beta+1-\frac{\varphi+1}{p}, 
\end{cases}
\end{equation}
or
\begin{equation}\label{ad-eq-2-1}
\begin{cases}
\lambda\geq \beta\geq 0, \\
\lambda>\alpha+\beta+1-\frac{\varphi+1}{p}.
\end{cases}
\end{equation}
then ${\bf H}_{\lambda, \alpha, \beta}$ is bounded from $l_{\varphi}^{p}$ to $L^{\infty}$.  We first prove that (\ref{ad-eq-1-1}) implies the boundedness of ${\bf H}_{\lambda, \alpha, \beta}$. Note that when (\ref{ad-eq-1-1}) holds, we have  $\lambda>\alpha+1-\frac{\varphi+1}{p}$. Then, for $a=\{a_m\}_{m=1}^{\infty}\in l_{\varphi}^p$ with $a_m\geq 0$ for $m\in \mathbb{N}$, by using H\"{o}lder's inequality, we have
\begin{eqnarray}\label{a-1}
\|{\bf H}_{\lambda, \alpha, \beta}a\|_{\infty}&=&\sup_{x\in \mathbb{R}_{+}}\sum_{m=1}^{\infty}\frac{m^{\alpha}x^{\beta}}{[\max\{m,x\}]^{\lambda}}a_m \\
&=&\sup_{x\in \mathbb{R}_{+}}\sum_{m=1}^{\infty}\frac{m^{\alpha-\frac{\varphi}{p}}x^{\beta}}{[\max\{m,x\}]^{\lambda}}\cdot[a_m m^{\frac{\varphi}{p}}] \nonumber \\
&\leq & \sup_{x\in \mathbb{R}_{+}}\Big\{\sum_{m=1}^{\infty}\frac{m^{p'(\alpha-\frac{\varphi}{p})}x^{p'\beta}}{[\max\{m,x\}]^{p'\lambda}}\Big\}^{\frac{1}{p'}}\cdot\|a\|_{p, \varphi}\nonumber \\
&=& \|a\|_{p, \varphi}\sup_{x\in \mathbb{R}_{+}}x^{\beta}\Big\{\sum_{m=1}^{\infty}\frac{m^{p'(\alpha-\frac{\varphi}{p})}}{[\max\{m,x\}]^{p'\lambda}}\Big\}^{\frac{1}{p'}}.\nonumber
\end{eqnarray}

First, we see from $p'(\alpha-\frac{\varphi}{p})-p'\lambda<-1$, i.e., $\lambda>\alpha+1-\frac{\varphi+1}{p}$, that
\begin{equation}\label{a-2}
\sup_{x\in (0,1]}x^{\beta}\Big\{\sum_{m=1}^{\infty}\frac{m^{p'(\alpha-\frac{\varphi}{p})}}{[\max\{m,x\}]^{p'\lambda}}\Big\}^{\frac{1}{p'}}=\sup_{x\in (0,1]}x^{\beta}\Big\{\sum_{m=1}^{\infty}{m^{p'(\alpha-\frac{\varphi}{p})-p'\lambda}}\Big\}^{\frac{1}{p'}}<\infty. 
\end{equation}

Second, when $-1<p'(\alpha-\frac{\varphi}{p})$, from Lemma \ref{ll-5} and 
$$\alpha-\frac{\varphi}{p}+\beta+\frac{1}{p'}-\lambda=\alpha+\beta+1-\frac{\varphi+1}{p}-\lambda\leq 0,$$ we have
\begin{eqnarray}\label{a-3}
\sup_{x>1}x^{\beta}\Big\{\sum_{m=1}^{\infty}\frac{m^{p'(\alpha-\frac{\varphi}{p})}}{[\max\{m,x\}]^{p'\lambda}}\Big\}^{\frac{1}{p'}}&\leq& C_1\sup_{x>1}x^{\beta}x^{\frac{1}{p'}[p'(\alpha-\frac{\varphi}{p'})+1-p'\lambda]}
  \\
&=&C_1\sup_{x>1}x^{\alpha-\frac{\varphi}{p}+\beta+\frac{1}{p'}-\lambda}<\infty.\nonumber
\end{eqnarray}

Also, when $p'(\alpha-\frac{\varphi}{p})\leq -1$, note that $\beta-\lambda<0$ and $p'\lambda-1-p'(\alpha-\frac{\varphi}{p})>0$, then from Lemma \ref{ll-5-1} we obtain that
 \begin{eqnarray}\label{a-4}
\lefteqn{\sup_{x>1}x^{\beta}\Big\{\sum_{m=1}^{\infty}\frac{m^{p'(\alpha-\frac{\varphi}{p})}}{[\max\{m,x\}]^{p'\lambda}}\Big\}^{\frac{1}{p'}}}
\nonumber \\
&&=\sup_{x>1}\Big\{x^{p'\beta}\sum_{m=1}^{\infty}\frac{m^{p'(\alpha-\frac{\varphi}{p})}}{[\max\{m,x\}]^{p'\lambda}}\Big\}^{\frac{1}{p'}}
\leq C_2^{\frac{1}{p'}}<\infty. 
\end{eqnarray}
Here, $C_2>0$ is independent of $x$. Consequently, we see from (\ref{a-1})-(\ref{a-4}) that there is a constant $C_3>0$ such that $\|{\bf H}_{\lambda, \alpha, \beta}a\|_{\infty}\leq C_3 \|a\|_{p, \varphi}$ so that ${\bf H}_{\lambda, \alpha, \beta}$ is bounded from $l_{\varphi}^{p}$ to $L^{\infty}$. 

We next prove that  (\ref{ad-eq-2-1}) implies the boundedness of ${\bf H}_{\lambda, \alpha, \beta}$. Note that when (\ref{ad-eq-2-1}) holds, we also have  $\lambda>\alpha+1-\frac{\varphi+1}{p}$. Then, for $a=\{a_m\}_{m=1}^{\infty}\in l_{\varphi}^p$ with $a_m\geq 0$ for $m\in \mathbb{N}$, repeating the above arguments, we can obtain that 
\begin{eqnarray}\label{x-1}
\|{\bf H}_{\lambda, \alpha, \beta}a\|_{\infty}\leq \|a\|_{p, \varphi}\sup_{x\in \mathbb{R}_{+}}x^{\beta}\Big\{\sum_{m=1}^{\infty}\frac{m^{p'(\alpha-\frac{\varphi}{p})}}{[\max\{m,x\}]^{p'\lambda}}\Big\}^{\frac{1}{p'}}, 
\end{eqnarray}
and
\begin{equation}\label{x-2}
\sup_{x\in (0,1]}x^{\beta}\Big\{\sum_{m=1}^{\infty}\frac{m^{p'(\alpha-\frac{\varphi}{p})}}{[\max\{m,x\}]^{p'\lambda}}\Big\}^{\frac{1}{p'}}<\infty, 
\end{equation}
and when $-1<p'(\alpha-\frac{\varphi}{p})$, 
\begin{eqnarray}\label{x-3}
\sup_{x>1}x^{\beta}\Big\{\sum_{m=1}^{\infty}\frac{m^{p'(\alpha-\frac{\varphi}{p})}}{[\max\{m,x\}]^{p'\lambda}}\Big\}^{\frac{1}{p'}}<\infty. 
\end{eqnarray}
Furthermore, if $\lambda<\beta$, then when $p'(\alpha-\frac{\varphi}{p})\leq -1$, as above, we can obtain that 
 \begin{eqnarray}\label{x-4}
\sup_{x>1}x^{\beta}\Big\{\sum_{m=1}^{\infty}\frac{m^{p'(\alpha-\frac{\varphi}{p})}}{[\max\{m,x\}]^{p'\lambda}}\Big\}^{\frac{1}{p'}}<\infty. 
\end{eqnarray}
Meanwhile, if $\lambda=\beta$, then from (\ref{ad-eq-2-1}), we know that $\alpha+1-\frac{\varphi+1}{p}<0,$ that is $p'(\alpha-\frac{\varphi}{p})<-1$. It follows from Lemma \ref{ll-5-1} again that (\ref{x-4}) still hold. 
Consequently, the boundedness of ${\bf H}_{\lambda, \alpha, \beta}$ follows from (\ref{x-1})-(\ref{x-4}). This proves Theorem \ref{m-th-3}. 

\section{{\bf Proof of Theorem \ref{m-th-3-1}}}
From Lemma \ref{ll-2}, we know that  ${\bf H}_{\lambda, \alpha, \beta}$ is bounded from $l^{\infty}$ to $L^{\infty}$ if and only if 
\begin{equation}\label{noa}\sup_{x\in \mathbb{R}_{+}}\sum_{m=1}^{\infty}\frac{m^{\alpha}x^{\beta}}{[\max\{m,x\}]^{\lambda}}<\infty.\end{equation}
Hence it is enough to prove that (\ref{noa}) holds if and only if 
\begin{equation}\label{ad-e-2}
\begin{cases}
\lambda\geq\beta\geq 0, \\
\lambda>\alpha+\beta+1.
\end{cases}
\end{equation}
We suppose that (\ref{noa}) holds, then from Lemma \ref{ll-5},  we first know that $\lambda-\alpha>1.$ 
Furthermore, when $x\in (0,1]$, we have 
\begin{equation}\label{ad-e-3}\infty >\sup_{x\in (0,1]}\sum_{m=1}^{\infty}\frac{m^{\alpha}x^{\beta}}{[\max\{m,x\}]^{\lambda}}=\sup_{x\in (0,1]}x^{\beta}\sum_{m=1}^{\infty}m^{\alpha-\lambda}.\nonumber  \end{equation}
This implies that $\beta\geq 0.$ Meanwhile, when $x>1$, from Remark \ref{f-re}, we know that 
\begin{equation}\label{ad-e-4}\infty>\sup_{x>1}\sum_{m=1}^{\infty}\frac{m^{\alpha}x^{\beta}}{[\max\{m,x\}]^{\lambda}}\geq C_1\sup_{x>1}x^{\beta}(x+1)^{\alpha+1-\lambda}\geq C_1\sup_{x>1}x^{\alpha+\beta+1-\lambda}. \nonumber\end{equation}
Here $C_1=(\lambda-\alpha-1)^{-1}$.  This implies that $\alpha+\beta+1-\lambda\leq 0$, that is $\lambda\geq \alpha+\beta+1.$ 
Also, when $x>1$, 
$$\infty>\sup_{x>1}x^{\beta}\sum_{m=1}^{\lceil x \rceil}\frac{m^{\alpha}}{{x}^{\lambda}}\geq \sup_{x>1}x^{\beta-\lambda}.$$
This implies that $\beta\leq \lambda$. Collecting the arguments above, we see that, if (\ref{noa}) holds, then  
\begin{equation}\label{ad-e-5}
\begin{cases}
\lambda\geq \beta\geq 0,\\ 
\lambda>\alpha+1, \\
\lambda\geq \alpha+\beta+1, 
\end{cases}
\end{equation}
which is equivalent to (\ref{ad-e-2}). This proves that (\ref{noa}) implies that (\ref{ad-e-2}) holds. 

Now, we assume that (\ref{ad-e-2}) holds, then we easily see that 
$$\sup_{x\in (0,1]}\sum_{m=1}^{\infty}\frac{m^{\alpha}x^{\beta}}{[\max\{m,x\}]^{\lambda}}=\sup_{x\in (0,1]}x^{\beta}\sum_{m=1}^{\infty}m^{\alpha-\lambda}<\infty.$$
On the other hand, when $\alpha>-1$, by Lemma \ref{ll-5}, we know that there is a constant $C_2>0$ such that  
$$\sup_{x>1}\sum_{m=1}^{\infty}\frac{m^{\alpha}x^{\beta}}{[\max\{m,x\}]^{\lambda}}\leq C_2 \sup_{x>1}x^{\alpha+\beta+1-\lambda}<\infty.$$
Meanwhile, when $\alpha\leq -1$, note that $\beta<\lambda$ in this case, by Lemma \ref{ll-5-1}, we know that there is a constant $C_3>0$ such that  
$$\sup_{x>1}\sum_{m=1}^{\infty}\frac{m^{\alpha}x^{\beta}}{[\max\{m,x\}]^{\lambda}}\leq C_3 \sup_{x>1}x^{\alpha+\beta+1-\lambda}<\infty.$$
It follows that (\ref{ad-e-2})) implies (\ref{noa}). Theorem \ref{m-th-3-1} is proved.

\section{{\bf Proof of Theorem \ref{th-r-10}, \ref{th-r-1} and \ref{th-r-2}}}
\subsection{{\bf Proof of Theorem \ref{th-r-10}}}
The only if part has been implied in Lemma \ref{ll-4-1}. We will prove the if part by using the duality. For any $g\in L^{\infty},$
we have 
\begin{eqnarray}\label{jia-e-1} \lefteqn{\int_{\mathbb{R}_{+}}g(x)x^{\psi}\Big[\sum_{m=1}^{\infty}\frac{m^{\alpha}x^{\beta}}{[\max\{m, x\}]^{\lambda}}a_m\Big]dx} \nonumber \\
&& \leq \|g\|_{\infty} \sum_{m=1}^{\infty}\Big[\int_{\mathbb{R}_{+}}\frac{m^{\alpha}x^{\beta+\psi}}{[\max\{m, x\}]^{\lambda}}dx\Big]|a_m| \nonumber \\
&& =\|g\|_{\infty} (\frac{1}{\beta+1+\psi}+\frac{1}{\beta+\psi-\lambda+1}) \sum_{m=1}^{\infty} m^{\alpha+\beta+\psi-\lambda+1}|a_m|.
\end{eqnarray}
Meanwhile, by H\"older's inequality, we have 
\begin{eqnarray}\label{jia-e-2}
 \sum_{m=1}^{\infty} m^{\alpha+\beta+\psi-\lambda+1}|a_m|&=&\sum_{m=1}^{\infty} m^{\alpha+\beta+\psi-\lambda+1-\frac{\varphi}{p}}|a_m|m^{\frac{\varphi}{p}} \nonumber \\
 &\leq& \Big[\sum_{m=1}^{\infty} m^{p'(\alpha+\beta+\psi-\lambda+1+\frac{\varphi}{p})}\Big]^{1/p'} \|a\|_{p, \varphi}. 
\end{eqnarray}
As $\lambda>\alpha+\beta+2+\psi-\frac{\varphi}{p}-\frac{1}{p}$, so $$p'(\alpha+\beta+\psi-\lambda+1-\frac{\varphi}{p})<p'(\frac{1}{p}-1)=-1.$$
This implies that  $$\sum_{m=1}^{\infty} m^{p'(\alpha+\beta+\psi-\lambda+1+\frac{\varphi}{p})}<+\infty.$$
It follows from (\ref{jia-e-1}) and  (\ref{jia-e-2}) that $$\int_{\mathbb{R}_{+}}g(x)x^{\psi}\Big[\sum_{m=1}^{\infty}\frac{m^{\alpha}x^{\beta}}{[\max\{m, x\}]^{\lambda}}a_m\Big]dx<+\infty,$$
for any $g\in L^{\infty}$. This means that  $\mathbf{H}_{\lambda, \alpha, \beta}$ is bounded from
$l_{\varphi}^p$ to $L_{\psi}^q$. This proves Theorem \ref{th-r-10}.

\subsection{{\bf Proof of Theorem \ref{th-r-1}}}
Since the only if part has been proved by Lemma \ref{ll-4-1}. We only need to prove that, if $\varphi+1<p(\alpha+1)$ and \begin{equation}\label{good-1}
\begin{cases}
\lambda>\alpha+\beta+1+\frac{\psi+1}{q}-\frac{\varphi+1}{p}, \\
-q\beta<\psi+1,  
\end{cases}
\end{equation}
then  $\mathbf{H}_{\lambda, \alpha, \beta}$ is bounded from $l_{\varphi}^{p}$ to $L_{\psi}^{q}$. Note that  $\mathbf{H}_{\lambda, \alpha, \beta}$ is bounded from $l_{\varphi}^{p}$ to $L_{\psi}^{q}$ is equivalent to the operator
$$\widehat{{\bf H}}_{\lambda, \alpha, \beta}a(x):=\sum_{m=1}^{\infty} \frac{m^{\alpha-\frac{\varphi}{p}}x^{\beta+\frac{\psi}{q}}}{[\max\{m,x\}]^{\lambda}}a_m,\,\, a=\{a_m\}_{m=1}^{\infty}, \,\, n\in \mathbb{N},$$
is bounded from $l^{p}$ to $L^{q}$. Now, for $\varepsilon>0$, we take $u_{\varepsilon}=\{u_m\}_{m=1}^{\infty}$ with $u_{m}=m^{-\frac{1}{p}-\frac{\varepsilon}{p}}$. Then we have $\|u_{\varepsilon}\|_{p}^p=\frac{1}{\varepsilon}[1+o(1)]$ as $\varepsilon \rightarrow 0$.
We let 
$$Iu(n):=\Big(\int_{\mathbb{R}_{+}}\frac{n^{\alpha-\frac{\varphi}{p}}x^{\beta+\frac{\psi}{q}}}{[\max\{n,x\}]^{\lambda}}\Big(\sum_{m=1}^{\infty}\frac{m^{\alpha-\frac{\varphi}{p}}x^{\beta+\frac{\psi}{q}}}{[\max\{m,x\}]^{\lambda}}u_m\Big)^{q-1}dx\Big)^{p'-1},\,\, n\in \mathbb{N}.$$
Note from $\varphi+1<p(\alpha+1)$ and Lemma \ref{ll-5} that 
\begin{eqnarray}
\sum_{m=1}^{\infty}\frac{m^{\alpha-\frac{\varphi}{p}}x^{\beta+\frac{\psi}{q}}}{[\max\{m,x\}]^{\lambda}}u_m\leq C_1x^{\alpha+\beta+1-\frac{\varphi+1}{p}-\frac{\varepsilon}{p}+\frac{\psi}{q}-\lambda},\nonumber 
\end{eqnarray}
when $\varepsilon<p(\alpha+1)-\varphi-1$. Here $C_1>0$ is independent of $x$. Then it follows that 
\begin{equation}
Iu(n)\leq C_1^{(q-1)(p'-1)}\Big(\int_{\mathbb{R}_{+}}\frac{n^{\alpha-\frac{\varphi}{p}}x^{\beta+\frac{\psi}{q}}}{[\max\{n,x\}]^{\lambda}}x^{(q-1)[\alpha+\beta+1-\frac{\varphi+1}{p}-\frac{\varepsilon}{p}+\frac{\psi}{q}-\lambda]}dx\Big)^{p'-1}, \nonumber\\
\end{equation}
for $0<\varepsilon<p(\alpha+1)-\varphi-1$. Consequently, from Lemma \ref{ll-4}, we have 
\begin{equation}
Iu(n)\leq C_1^{(q-1)(p'-1)}C_2^{p'-1}n^{(p'-1)\{\alpha-\frac{\varphi}{p}+\beta+\frac{\psi}{q}+(q-1)[\alpha+\beta+1-\frac{\varphi+1}{p}-\frac{\varepsilon}{p}+\frac{\psi}{q}-\lambda]+1-\lambda\}}.\nonumber
\end{equation}
On the other hand, as $\lambda>\alpha+\beta+1+\frac{\psi+1}{q}-\frac{\varphi+1}{p}$, then we can write
$$\lambda=\delta+\alpha+\beta+1+\frac{\psi+1}{q}-\frac{\varphi+1}{p}.$$
Here $\delta>0$. It follows that, when $(1-\frac{q}{p})\varepsilon \leq q\delta$, i.e., $0<\varepsilon<pq\delta(p-q)^{-1}$,
\begin{eqnarray}
\lefteqn{(p'-1)\{\alpha-\frac{\varphi}{p}+\beta+\frac{\psi}{q}+(q-1)[\alpha+\beta+1-\frac{\varphi+1}{p}-\frac{\varepsilon}{p}+\frac{\psi}{q}-\lambda]+1-\lambda\}} \nonumber
\\
&&\qquad\qquad\quad=(p'-1)[q(\alpha+\beta+1-\lambda-\frac{\varphi}{p}+\frac{\psi}{q})-(q-1)\frac{1+\varepsilon}{p}] \nonumber \\
&&\qquad\qquad\quad=(p'-1)[q(\frac{1}{p}-\frac{1}{q}-\delta)-(q-1)\frac{1+\varepsilon}{p}]\nonumber \\
&&\qquad\qquad\quad=(p'-1)[-1+\frac{1+\varepsilon}{p}-\frac{q\varepsilon}{p}-q\delta]\nonumber \\
&&\qquad\qquad\quad\leq (p'-1)[-1+\frac{1+\varepsilon}{p}-\frac{q\varepsilon}{p}-(1-\frac{q}{p})\varepsilon]=-\frac{1+\varepsilon}{p}.\nonumber\end{eqnarray}
Hence we obtain that
$$ 
Iu(n)\leq C_1^{(q-1)(p'-1)}C_2^{p'-1}n^{-\frac{\varepsilon+1}{p}}=C_1^{(q-1)(p'-1)}C_2^{p'-1}u_n, 
$$
for all $n\in \mathbb{N}$, when $0<\varepsilon<\min\{p(\alpha+1)-\varphi-1,pq\delta(p-q)^{-1}\}.$

Consequently, from Lemma \ref{ll-3}, we know that the operator $\widehat{{\bf H}}_{\lambda, \alpha, \beta}$ is bounded from $l^{p}$ to $L^{q}$. That is to say  $\mathbf{H}_{\lambda, \alpha, \beta}$ is bounded from $l_{\varphi}^{p}$ to $L_{\psi}^{q}$.
This proves Theorem \ref{th-r-1}.

\subsection{{\bf Proof of Theorem \ref{th-r-2}}}
From Lemma \ref{ll-2}, we know that $\mathbf{H}_{\lambda, \alpha, \beta}$ is bounded from $l^{\infty}$ to $L_{\psi}^{q}$ if and only if 
\begin{equation}
{\bf I}:=\int_{\mathbb{R}_{+}}\Big[\sum_{m=1}^{\infty}\frac{m^{\alpha}x^{\beta+\frac{\psi}{q}}}{[\max\{m,x\}]^{\lambda}}\Big]^q dx=
\int_{\mathbb{R}_{+}}x^{q\beta+\psi}\Big[\sum_{m=1}^{\infty}\frac{m^{\alpha}}{[\max\{m,x\}]^{\lambda}}\Big]^q dx<\infty. \nonumber
\end{equation}
We let
$${\bf I}_1:=\int_{0}^{1}x^{q\beta+\psi}\Big[\sum_{m=1}^{\infty}\frac{m^{\alpha}}{[\max\{m,x\}]^{\lambda}}\Big]^qdx=\int_{0}^{1}x^{q\beta+\psi}\Big[\sum_{m=1}^{\infty}m^{\alpha-\lambda}\Big]^q dx,$$
$${\bf I}_2:=\int_{1}^{\infty}x^{q\beta+\psi}\Big[\sum_{m=1}^{\infty}\frac{m^{\alpha}}{[\max\{m,x\}]^{\lambda}}\Big]^q dx.$$
We suppose that $\mathbf{H}_{\lambda, \alpha, \beta}$ is bounded from $l^{\infty}$ to $L_{\psi}^{q}$, i.e ${\bf I}<\infty$. Then from ${\bf I}_1<\infty$, we have $\alpha-\lambda<-1$ and $q\beta+\psi>-1$.   
Furthermore, from the Remark \ref{f-re}, we know that, for each $x\geq 1$, 
\begin{equation}
\sum_{m=1}^{\infty}\frac{m^{\alpha}}{[\max\{m,x\}]^{\lambda}}\geq \frac{(x+1)^{\alpha+1-\lambda}}{\lambda-\alpha-1}\geq \frac{2^{\alpha+1-\lambda}}{\lambda-\alpha-1}x^{\alpha+1-\lambda}. \nonumber
\end{equation}
Hence we see from ${\bf I}_2<\infty$ that 
$$\int_{1}^{\infty}x^{q\beta+\psi+q(\alpha+1-\lambda)}dx<\infty,$$
so that $q\beta+\psi+q(\alpha+1-\lambda)<-1,$ which is equivalent to 
$$\lambda>\alpha+\beta+1+\frac{\psi+1}{q}.$$ This proves the only if part of Theorem \ref{th-r-2}.

Next, we will show that, if \begin{equation}\label{jia-1}
\begin{cases}
\alpha\geq -1,\\
\lambda>\alpha+\beta+1+\frac{\psi+1}{q}, \\
-q\beta<\psi+1, 
\end{cases}
\end{equation}
then ${\bf I}<\infty$. As $-q\beta<\psi+1$, i.e., $\beta+\frac{\psi+1}{q}>0$, then we get that
$$\lambda>\alpha+\beta+1+\frac{\psi+1}{q}>\alpha+1.$$
It follows that ${\bf I}_1<\infty$.  To prove $\mathbf{I}_2<\infty$, we divide our proof into two cases.

(1) When $\alpha>-1$, in view of $\lambda>\alpha+1$, by Lemma \ref{ll-5}, we know that there is a constant $C_3>0$ such that
$${\bf I}_2\leq C_3 \int_{1}^{\infty}x^{q\beta+\psi}x^{q(\alpha+1-\lambda)}dx=C_3\int_{1}^{\infty}x^{q\beta+\psi+q(\alpha+1-\lambda)}dx<\infty,$$
Consequently, ${\bf I}_2<\infty$ since $\lambda>\alpha+\beta+1+\frac{\psi+1}{q}$. 

(2) When $\alpha=-1$, since $\lambda>\alpha+\beta+1+\frac{\psi+1}{q}=\beta+\frac{\psi+1}{q}>0$, let $\tau:=\lambda-\beta-\frac{\psi+1}{q}>0$, then we see from the proof of Lemma \ref{ll-5-1} that there is a constant $\epsilon\in (0,\tau)$ such that
\begin{eqnarray}{\bf I}_2&\leq& C_4 \int_{1}^{\infty}x^{q\beta+\psi}x^{-q\lambda+q\epsilon}dx\nonumber \\
&=&C_4\int_{1}^{\infty}x^{q(\beta+\frac{\psi+1}{q}-\lambda)-1+q\epsilon}dx \nonumber \\
&\leq & C_4 \int_{1}^{\infty}x^{-1-q(\tau-\epsilon)}dx<\infty.\nonumber\end{eqnarray}

Combining (1) and (2), we obtain that $\mathbf{I}=\mathbf{I}_1+\mathbf{I}_2<\infty$ when (\ref{jia-1}) is satisfied. Now, the if part of Theorem \ref{th-r-2} is proved so that we finish the proof of Theorem \ref{th-r-2}. 

\section{{\bf Sharp norm estimates of $\mathbf{H}_{\lambda, \alpha, \beta}$ for some special cases}}
In this section, we will establish sharp norm estimates of $\mathbf{H}_{\lambda, \alpha, \beta}$ for certain special cases. We shall prove that
\begin{theorem}\label{last}
Let $1<p<\infty$. Let $\lambda, \alpha, \beta, \varphi, \psi$ be real numbers with
\begin{equation}\label{=c}\lambda=\alpha+\beta+1+\frac{\psi-\varphi}{p},\end{equation}
and ${\bf H}_{\lambda, \alpha, \beta}$ be as in (\ref{ope}). Then $\mathbf{H}_{\lambda, \alpha, \beta}$ is bounded from $l_{\varphi}^{p}$ to $L_{\psi}^{p}$ if and only if 
\begin{equation} \label{ineq-1}
-p\beta<\psi+1<p(\lambda-\beta),  
\end{equation}
or equivalently, 
\begin{equation}\label{ineq-2} 
p(\alpha+1-\lambda)<\varphi+1<p(\alpha+1). 
\end{equation}
Moreover, if the conditions (\ref{=c}), (\ref{ineq-1})(or (\ref{ineq-2})), and 
\begin{equation}\label{ineq-3}
\varphi+1\geq p\alpha,
\end{equation}
%\begin{equation}\label{ineq-3}
%\begin{cases}
%\varphi+1\geq p\alpha, \\
%\psi+1\leq p(1-\beta), 
%\end{cases}
%\end{equation}
are all satisfied, then $\mathbf{H}_{\lambda, \alpha, \beta}$ is bounded from $l_{\varphi}^{p}$ to $L_{\psi}^{p}$, and the norm of $\mathbf{H}_{\lambda, \alpha, \beta}$ is given by
 $$\|\mathbf{H}_{\lambda, \alpha, \beta}\|_{l_{\varphi}^{p}\rightarrow L_{\psi}^p}=\frac{1}{\alpha+1-\frac{1}{p}(\varphi+1)}+\frac{1}{\beta+\frac{1}{p}(\psi+1)}.$$ 
Here, $$\|\mathbf{H}_{\lambda, \alpha, \beta}\|_{l_{\varphi}^{p}\rightarrow L_{\psi}^p}=\sup_{a\in l_{\varphi}^{p}}\frac{\|\mathbf{H}_{\lambda,\alpha,\beta}a\|_{p, \psi}}{\|a\|_{p, \varphi}}.$$
\end{theorem}

\begin{proof}Note that if (\ref{=c}) holds, then it is easy to see that (\ref{ineq-1}) is equivalent to (\ref{ineq-2}). By Lemma \ref{ll-4} and repeating the arguments in the proof of {\bf Case ${\bf I}$} of Theorem \ref{m-th-1}, we can prove that ${\bf H}_{\lambda, \alpha, \beta}$ is bounded from $l_{\varphi}^{p}$ to $L_{\psi}^{p}$ if and only if (\ref{ineq-1})(or (\ref{ineq-2})) holds. Next, we will prove that if (\ref{=c}), (\ref{ineq-1}) and (\ref{ineq-3}) are all satisfied, then  
 $$\|{\bf H}_{\lambda, \alpha, \beta}\|_{l_{\varphi}^{p}\rightarrow l_{\psi}^p}=\frac{1}{\alpha+1-\frac{1}{p}(\varphi+1)}+\frac{1}{\beta+\frac{1}{p}(\psi+1)}.$$

We shall borrow some arguments in the proof of {\bf Case ${\bf I}$} of Theorem \ref{m-th-1}. Actually, when $p=q$, in the proof of {\bf Case ${\bf I}$} of Theorem \ref{m-th-1}, we can take $s=p', t=p$ and $A=-\frac{\varphi+1}{pp'}$. Then, when (\ref{ineq-3}) holds so that (\ref{con-1}) is satisfied, we have
$${\bf C}_1={\bf C}_2=\frac{1}{\alpha+1-\frac{1}{p}(\varphi+1)}+\frac{1}{\beta+\frac{1}{p}(\psi+1)}.$$
Consequently, we have $$\|{\bf H}_{\lambda, \alpha, \beta}\|_{l_{\varphi}^{p}\rightarrow L_{\psi}^p}\leq \frac{1}{\alpha+1-\frac{1}{p}(\varphi+1)}+\frac{1}{\beta+\frac{1}{p}(\psi+1)}.$$

Finally, we prove that  $$\|{\bf H}_{\lambda, \alpha, \beta}\|_{l_{\varphi}^{p}\rightarrow L_{\psi}^p}\geq \frac{1}{\alpha+1-\frac{1}{p}(\varphi+1)}+\frac{1}{\beta+\frac{1}{p}(\psi+1)}.$$ 
For $0<\varepsilon<p(\alpha+1)-(\varphi+1)$, we take the same $a_{\varepsilon}$ as in (\ref{ff-2}). Then \begin{equation}\label{jia}\|a_{\varepsilon}\|_{p, \alpha}^p=\frac{1}{\varepsilon}[1+o(1)],\,\, {\text {as}}\,\, \varepsilon \rightarrow 0^{+},\end{equation}
and from $\alpha-\frac{\varphi+1}{p}-\frac{\varepsilon}{p}<0$ we have
\begin{eqnarray}
\|{\bf H}_{\lambda, \alpha, \beta}a_{\varepsilon}\|_{p, \psi}&=&\Big[\int_{0}^{\infty}x^{\psi}\Big(\sum_{m=1}^{\infty}\frac{m^{\alpha}x^{\beta}}{[\max\{m,x\}]^{\lambda}}m^{-\frac{\varphi+1}{p}-\frac{\varepsilon}{p}}\Big)^p\Big]^{\frac{1}{p}} \nonumber 
\\
&\geq & \Big[\int_{0}^{\infty}x^{\psi}\Big(\int_{1}^{\infty}\frac{y^{\alpha}x^{\beta}}{[\max\{y,x\}]^{\lambda}}y^{-\frac{\varphi+1}{p}-\frac{\varepsilon}{p}}dy\Big)^p dx\Big]^{\frac{1}{p}} \nonumber \\
&=& \Big[\int_{0}^{\infty}x^{\psi+p(\alpha+\beta+1-\frac{\varphi+1}{p}-\frac{\varepsilon}{p}-\lambda)}\Big(\int_{\frac{1}{x}}^{\infty}\frac{t^{\alpha-\frac{\varphi+1}{p}-\frac{\varepsilon}{p}-\lambda}}{[\max\{t,1\}]^{\lambda}}dt\Big)^p dx\Big]^{\frac{1}{p}}\nonumber
\\
&=& \Big[\int_{0}^{\infty}x^{-\varepsilon-1}\Big(\int_{\frac{1}{x}}^{\infty}\frac{t^{\alpha-\frac{\varphi+1}{p}-\frac{\varepsilon}{p}}}{[\max\{t,1\}]^{\lambda}}dt\Big)^p dx\Big]^{\frac{1}{p}}\nonumber
\end{eqnarray} 
Now, we take $\varepsilon=\frac{1}{z}$ with $z>[p(\alpha+1)-(\varphi+1)]^{-1}$. Then we have 
\begin{eqnarray}
\|{\bf H}_{\lambda, \alpha, \beta}a_{\varepsilon}\|_{p, \psi}&\geq & \Big[\int_{z}^{\infty}x^{-\varepsilon-1}\Big(\int_{\frac{1}{z}}^{\infty}\frac{t^{\alpha-\frac{\varphi+1}{p}-\frac{\varepsilon}{p}}}{[\max\{t,1\}]^{\lambda}}dt\Big)^p dx\Big]^{\frac{1}{p}} \nonumber 
\\
&=& (\varepsilon^{\varepsilon-1})^{\frac{1}{p}}\int_{\frac{1}{z}}^{\infty}\frac{t^{\alpha-\frac{\varphi+1}{p}-\frac{\varepsilon}{p}}}{[\max\{t,1\}]^{\lambda}}dt.  \nonumber
\end{eqnarray} 
It follows that 
\begin{eqnarray}\label{m-s} \|\mathbf{H}_{\lambda, \alpha, \beta}\|_{l_{\varphi}^p \rightarrow L_{\psi}^p} \geq \frac{\|{\bf H}_{\lambda, \alpha, \beta}a_{\varepsilon}\|_{p, \psi}}{\|a_{\varepsilon}\|_{p, \alpha}}
\geq {\varepsilon}^{\frac{\varepsilon}{p}}[1+o(1)]^{-\frac{1}{p}}\int_{\frac{1}{z}}^{\infty}\frac{t^{\alpha-\frac{\varphi+1}{p}-\frac{\varepsilon}{p}}}{[\max\{t,1\}]^{\lambda}}dt.\end{eqnarray}
Here $o(1)$ is the same as in (\ref{jia}). We let 
$$\mathcal{S}_{{z}}=\{t\in \mathbb{R}_{+}: t \geq \frac{1}{{z}}\}.$$
Then 
\begin{eqnarray}\label{l-s}\int_{\frac{1}{z}}^{\infty}\frac{t^{\alpha-\frac{\varphi+1}{p}-\frac{1}{{z}p}}}{[\max\{t,1\}]^{\lambda}}dt
=\int_{0}^{\infty}\frac{\chi_{\mathcal{S}_{z}}(t)t^{\alpha-\frac{\varphi+1}{p}-\frac{1}{{z}p}}}{[\max\{t,1\}]^{\lambda}}dt. 
\end{eqnarray}
Note that, for each $t\in \mathbb{R}_{+}$,
\begin{eqnarray}\label{n-s}\frac{\chi_{\mathcal{S}_{z}}(t)t^{\alpha-\frac{\varphi+1}{p}-\frac{1}{{z}p}}}{[\max\{t,1\}]^{\lambda}}\rightarrow \frac{t^{\alpha-\frac{\varphi+1}{p}}}{[\max\{t,1\}]^{\lambda}},\,\, {\text as}\,\, {z}\rightarrow \infty,
\end{eqnarray}
and $ {\varepsilon}^{\frac{\varepsilon}{p}}[1+o(1)]^{-\frac{1}{p}}\rightarrow 1$ as ${z}\rightarrow \infty.$ 

Consequently, from (\ref{m-s}), (\ref{l-s}), (\ref{n-s}) and by using Fatou's lemma, we obtain that 
\begin{eqnarray*} \lefteqn{\|\mathbf{H}_{\lambda, \alpha, \beta}\|_{l_{\varphi}^p \rightarrow L_{\psi}^p} \geq \liminf_{z\to \infty} {\varepsilon}^{\frac{\varepsilon}{p}}[1+o(1)]^{-\frac{1}{p}}\int_{0}^{\infty}\frac{\chi_{\mathcal{S}_{z}}(t)t^{\alpha-\frac{\varphi+1}{p}-\frac{1}{{z}p}}}{[\max\{t,1\}]^{\lambda}}dt}\nonumber \\
&&\geq 
\int_{0}^{\infty} \frac{t^{\alpha-\frac{\varphi+1}{p}}}{[\max\{t,1\}]^{\lambda}}\,dt=\frac{1}{\alpha+1-\frac{1}{p}(\varphi+1)}+\frac{1}{\beta+\frac{1}{p}(\psi+1)}. \end{eqnarray*}
This proves Theorem \ref{last}.
\end{proof} 
\section{{\bf Final remarks}}
\begin{remark}We first point out that the condition $\alpha+1\geq 0$ in Theorem \ref{th-r-2} is necessary. As a counterexample, we consider the case $\alpha<-1$, $\lambda=0$ and $q, \psi$ satisfying that 
$0<\beta+\frac{\psi+1}{q}<-\alpha-1$.
Then 
\begin{eqnarray}
{\bf I}&=&\int_{\mathbb{R}_{+}}\Big[\sum_{m=1}^{\infty}\frac{m^{\alpha}x^{\beta+\frac{\psi}{q}}}{[\max\{m,x\}]^{\lambda}}\Big]^q dx\nonumber
\\&=&
\int_{\mathbb{R}_{+}}x^{q\beta+\psi}\Big[\sum_{m=1}^{\infty}m^{\alpha}\Big]^q dx=C(\alpha, q)\int_{\mathbb{R}_{+}}x^{q\beta+\psi}dx. \nonumber
\end{eqnarray}
But we always have \[\int_{\mathbb{R}_{+}}x^{q\beta+\psi}dx=+\infty, \,{\text{ for any}}\, q,\beta, \psi.\] This means that the operator is not bounded in this case and therefore the condition in Theorem \ref{th-r-2} can not be removed. \end{remark}

\begin{remark}
We finally show by another counterexample that the condition 
\(\varphi+1 < p(\alpha+1)\) in Theorem \ref{th-r-1} is necessary and can not be removed.  

We take $p=3, q=2, \alpha=-1/2,\beta=0,\varphi=2, \psi=0,\lambda=1/2$. Then 
$\varphi+1=3>p(\alpha+1)=3/2$, and  
\[
0=-q\beta<\psi+1=1,
\]
and
\[
\frac{1}{2}=\lambda>\alpha+\beta+1+\frac{\psi+1}{q}-\frac{\varphi+1}{p}=0.
\]
Now, for \(\varepsilon>0\), we consider the sequence $a_{\varepsilon}=\{a_m\}_{m=1}^{\infty}$ with 
\[a_m = m^{-\frac{\varphi+1}{p}-\frac{\varepsilon}{p}} = m^{-1-\frac{\varepsilon}{2}},\, m\in \mathbb{N}.\]
It is easy to see that
\[
\|a\|_{p,\varphi}^p=\sum_{m=1}^\infty m^{-1-\varepsilon}< \infty,
\]
so that \(a_{\varepsilon} \in l_\varphi^p\). Note that for \(x>1\),
\[
\mathbf{H}_{\lambda,\alpha,\beta}a_{\varepsilon}(x)=\sum_{m=1}^\infty \frac{m^{-3/2-\varepsilon/2}}{[\max\{m,x\}]^{1/2}}.
\]
It follows that for any $x>1$,
\[
\mathbf{H}_{\lambda,\alpha,\beta}a_{\varepsilon}(x)\geq\sum_{m=1}^{\lceil x \rceil}\frac{m^{-3/2-\varepsilon/2}}{x^{1/2}}\geq x^{-1/2}\sum_{m=1}^{\lceil x \rceil}m^{-3/2-\varepsilon/2}\geq x^{-1/2},
\]
Consequently, we obtain that 
\[
\|\mathbf{H}_{\lambda,\alpha,\beta}a_{\varepsilon}\|_{q,\psi}^q=\int_0^\infty x^\psi |\mathbf{H}_{\lambda,\alpha,\beta}a_{\varepsilon}(x)|^q dx
\geq \int_1^\infty x^{-1} dx=+\infty.
\]
Hence $\mathbf{H}_{\lambda,\alpha,\beta}$ is not bounded from \(l_\varphi^p\) to \(L_\psi^q\), although all the conditions of Theorem \ref{th-r-1} are satisfied except the one \(\phi+1 < p(\alpha+1)\). This means that 
the condition \(\varphi+1 < p(\alpha+1)\) in Theorem \ref{th-r-1} can not be removed. 
\end{remark}

\section{{\bf Acknowledgements}}
The author was supported by the National Natural Science Foundation of China (Grant No. 11501157).
\begin{spacing}{1.2} % 1.5倍行距

\end{spacing}
\end{document}